\documentclass{aims}
\usepackage{color}
\usepackage{amsmath}
  \usepackage{paralist}
  \usepackage{graphics} 
  \usepackage{epsfig} 
\usepackage{graphicx}  \usepackage{epstopdf}
 \usepackage[colorlinks=true]{hyperref}
\hypersetup{urlcolor=blue, citecolor=red}

\def\currentvolume{X}
 \def\currentissue{X}
  \def\currentyear{200X}
   \def\currentmonth{XX}
    \def\ppages{X--XX}
     \def\DOI{10.3934/xx.xx.xx.xx}

\newtheorem{theorem}{Theorem}[section]

\newtheorem{lemma}[theorem]{Lemma}

\theoremstyle{definition}
\newtheorem{definition}[theorem]{Definition}
\newtheorem{remark}{Remark}

\title[Computer Assisted Proofs of Attracting Invariant Tori for ODEs] 
      {Computer Assisted Proofs of Two-Dimensional Attracting Invariant Tori for ODEs}

\author[M.J. Capi\'nski, E. Fleurantin and J.D. Mireles James]{}

\subjclass{34C45, 
 70K43, 
 37G35, 
 65P20, 
 65G20 
 }
 \keywords{Attracting invariant tori, computer assisted proof, dynamics of ordinary differential equations}

 \email{maciej.capinski@agh.edu.pl}
 \email{efleurantin2013@fau.edu}
 \email{jmirelesjames@fau.edu}

\thanks{The first author is supported by by the NCN grant  2018/29/B/ST1/00109 and by the Faculty of Applied Mathematics AGH UST statutory tasks 11.11.420.004 within subsidy of Ministry of Science and Higher Education. The work has been conducted during the visit to FAU sponsored by the Fulbright Foundation. 
The third author was partially supported by  National Science Foundation grant DMS 1813501. }

\thanks{$^*$ Corresponding author: maciej.capinski@agh.edu.pl }

\newcommand{\correction}[2]{#2}
\begin{document}

\maketitle

\centerline{\scshape Maciej J. Capi\'nski$^*$}
\medskip
{\footnotesize
 \centerline{AGH University of Science and Technology}
   \centerline{ al. Mickiewicza 30, 30-059 Krak\'ow, Poland }
} 

\medskip

\centerline{\scshape Emmanuel Fleurantin and J.D. Mireles James}
\medskip
{\footnotesize
 \centerline{ Florida Atlantic University}
   \centerline{777 Glades Rd., Boca Raton, FL 33431, USA}
}

\bigskip

 \centerline{(Communicated by the associate editor name)}

\begin{abstract}
This work studies existence and regularity questions for attracting invariant tori in three dimensional dissipative systems of ordinary differential equations. Our main result is a constructive method of computer assisted proof which applies to explicit problems in non-perturbative regimes.  We obtain verifiable lower bounds on the regularity of the attractor in terms of the ratio of the expansion rate on the torus with the contraction rate near the torus.  We consider separately two important cases of rotational and resonant tori.  
In the rotational case we obtain $C^k$ lower bounds on the regularity of the embedding.  
In the resonant case we verify the existence of tori which are only $C^0$ and neither star-shaped nor Lipschitz.   
\end{abstract}


\section{Introduction} \label{sec:intro}
Questions about the existence, topology, and regularity
of invariant sets have organized the qualitative theory of 
nonlinear dynamics since the foundational work of Poincar\'e at the 
end of the Nineteenth Century.  In modern times numerical simulations 
play a crucial role in this theory, providing deeper insights into the
fine structure of phase space than can be obtained by any other means.  
The digital computer
has emerged as a kind of dynamical systems laboratory, where one runs
experiments on nonlinear systems far from a trivial solution or other 
perturbative regime.     

In response to this development last four decades have seen 
a number of researchers put
tremendous energy into developing and deploying computer assisted
methods of proof which bridge the gap between numerical conjecture
and mathematically rigorous theorems.    The 
 work of Lanford, Eckman, and Collet on the computer assisted 
proof of the Feigenbaum Conjectures \cite{lanford_CAP,compProofFeig},
and the resolution of Smale's 14th problem by Tucker \cite{tucker1,tucker2}
provide excellent examples of this trend.
The recent review articles \cite{reviewCAP,MR3990999} provide historical context 
and more complete discussion of the literature.

The present work focuses on computer assisted methods of proof for 
attracting invariant tori in dissipative vector fields.  
Invariant tori typically appear in systems where there are two or more competing 
natural frequencies.  Two common mechanisms are periodic/quasi-periodic 
perturbations of a system with an attracting periodic orbit, and when a 
periodic orbit with complex conjugate Floquet multipliers loses stability -- 
triggering a 
\correction{comment 17}{Neimark-Sacker} 
bifurcation in a Poincar\'e section
\cite{MR0132256,MR2615427}.  Both situations are treated in the 
present work.  Some classic references on dissipative dynamical systems 
having invariant tori are 
\cite{MR1005055,MR1115870,MR709899,MR880159,MR1488520,MR3435117,MR3279518,MR1093209,Langford},
and we refer also to the works of 
\cite{Haro1,Haro2,MR2299977,MR3713932,MR3713933,MR3309008}
for a functional analytic approach to this topic.  
Of course the study of robust invariant manifolds, or \textit{normally 
hyperbolic invariant manifolds} (NHIMs) goes back to the classic works 
of Fenichel \cite{MR287106}, and of Hirsch, Pugh, and Shub \cite{MR0271991,MR292101,MR0501173}.
See also the works of 
\cite{MR754826,MR791842,MR1391508,MR1460262,MR2136745,MR3751167}
and the references therein for some numerical investigations of NHIMs.

Related techniques for computer assisted proof of invariant tori are 
found in the works of  
\cite{celettiCAP1,cellettiCAP2,rana1,rana2,CAPmodernKAM,Haro_aPosKAM},
and the references therein.
It should be remarked that the works just cited deal with analytic invariant 
KAM tori in symplectic/Hamiltonian systems, where the torus 
cannot be attracting and the dynamics 
on the torus are conjugate to a Diophantine irrational rotation.
The present work deals with attracting invariant 
tori in dissipative systems. These objects are necessarily of 
lower regularity \cite{Llave} -- $C^k$ sometimes with $0 \leq k < \infty$ -- 
and computer assisted existence proofs require  
different strategies.

Our analysis is formulated in terms of topological and geometric hypotheses 
which we check using mathematically rigorous computational 
techniques for numerical integration of vector fields and their variational 
equations.  To make the presentation 
as self contained as possible we focus on the case of 3D fields and include 
elementary proofs of our arguments.
We implement our method in two illustrative examples.  The first example
is a periodic perturbation of a planar vector field where the unperturbed
system has an attracting periodic orbit.  Here we prove the existence
of $C^k$ invariant tori with rotational dynamics.  The second example is a 
is an autonomous vector field where resonant invariant tori appear naturally after a 
 \correction{comment 17}{Neimark-Sacker bifurcation}. 

The two situations require different analysis.  
In the rotational case we develop an \textit{outer approximation} 
of the torus via coverings by polygons and cone conditions. The union of 
the polygons is eventually shown to contain a torus.  
In the case of the autonomous vector field, where the invariant torus appears in a 
 \correction{comment 17}{Neimark-Sacker bifurcation}, the tori we consider 
 are resonant.  This means that they 
can be decomposed into attracting and saddle periodic orbits, where the unstable 
manifold of the saddle is absorbed completely into a trapping neighborhood 
of the attracting orbit.  In the resonant case we provide an \textit{inner approximation}, 
in the sense that we build the torus out of invariant pieces whose union  
is shown to be the desired torus.  
In both the rotational and resonant cases we study the tori away from the 
perturbative case.  
We remark that, because our theoretical arguments are formulated for maps 
(in our case Poincar\'{e} maps),  
our implementations rely heavily on the validated 
$C^k$ integrators developed over the last decade by Wilczak and Zgliczy\'nski
\cite{cnLohner,c1Lohner}.

A technical remark is that our analysis of the rotational tori 
is formulated for a star-shaped region in an appropriate surface of section.
The star-shaped hypothesis is an implementation detail which allows us to proceed without 
making a technical digression into the setting of vector bundles. 
Nevertheless, we indicate in Section \ref{sec:generalization}
how to proceed more generally. 
\correction{comment 35}{
A second technical remark is that the torus may be smoother than we are actually
able to prove.   Put another way, we prove that the rotational torus is \textit{at least} $C^k$ though it may in fact 
be smoother.  We do not claim that our regularity results are sharp.}
On the other hand the resonant tori we study are globally only $C^0$, and in this case the 
regularity is sharp as the tori are not globally Lipschitz.  

The remainder of the paper is organized as follows.
In Section \ref{sec:prelims} we review some preliminary 
notions and definitions from dynamics and validated numerics.  
In Section \ref{sec:contractiing-maps}
we state and prove our main theorem on 
the existence of attracting invariant Lipschitz curves, and investigate conditions which imply their differentiability. 

Section \ref{sec:homoclinic-maps} treats computer assisted 
methods of validation for the existence of homoclinic/heteroclinic
orbits for planar maps.  More explicitly we develop techniques for
 proving the existence of attracting fixed points and obtaining
lower bounds on the size of the basin of attraction.  Then 
we recall some tools for 
validating bounds on the local stable/unstable manifolds attached to
saddle fixed points from \cite{Cap-Lyap,Zgliczynski-cone-cond}.
Finally we prove the existence of heteroclinic connections from the 
saddle to the attractor.  When these techniques are applied in a 
Poincar\'e section for an ODE we obtain connections between periodic 
orbits of the differential equation.

In Section \ref{sec:tori-3d} we show how to apply the methods of Section  \ref{sec:homoclinic-maps}
to prove the existence of invariant tori for ODEs.   The main idea is to propagate an invariant circle from a Poincar\'e section by the flow of the ODE.

Section \ref{sec:examples} is devoted to the implementation 
\correction{comment 8}{of} 
 our methods
in two example applications.  We consider a periodically forced Van der Pol 
equation where the natural attracting periodic orbit in the unforced system 
gives an attracting invariant torus after the application of the forcing. 
Here we prove the existence of $C^k$ invariant tori.
We also consider an autonomous differential equation with an 
attracting resonant tori.  There is an attracting periodic orbit in the 
invariant the torus which has complex conjugate multipliers, 
hence the torus is only $C^0$.


\section{Preliminaries} \label{sec:prelims}

For a set $A$ we write $\overline{A}$ to denote its closure,
$\mathrm{int}A$ to denote its interior, and write $\mathbb{S}^{1}$ for a
one dimensional circle. 
\correction{comment 18}{Throughout the paper, for $x\in \mathbb{R}^n$ we shall use $\| x \|$ to stand for the Euclidean norm.}
Let $B(p,r)$ denote the
open ball of radius $r$ centered at $p$.

Suppose that $f:\mathbb{R}^{n}\rightarrow \mathbb{R}^{n}$ 
is a diffeomorphism and let $%
p^{\ast }\in \mathbb{R}^{n}$ be a hyperbolic fixed point of $f$ (i.e. the
eigenvalues of $Df(p^{\ast })$ are \correction{comment 9}{not on} the unit circle). We shall
use the notation $W^{u}(p^{\ast })$ and $W^{s}(p^{\ast })$ to stand for the
unstable and the stable manifold of $p^{\ast }$, respectively, i.e. 
\begin{eqnarray*}
W^{u}\left( p^{\ast }\right)  &=&\left\{ p:\left\Vert f^{n}(p)-p^{\ast
}\right\Vert \rightarrow 0\text{ as }n\rightarrow -\infty \right\} , \\
W^{s}\left( p^{\ast }\right)  &=&\left\{ p:\left\Vert f^{n}(p)-p^{\ast
}\right\Vert \rightarrow 0\text{ as }n\rightarrow \infty \right\} .
\end{eqnarray*}

For $f:\mathbb{R}^{n}\rightarrow\mathbb{R}^{n}$ and $B\subset\mathbb{R}^{n}$
define $\left[ Df\left( B\right) \right] \subset\mathbb{R}^{n}\times%
\mathbb{R}^{n}$ as%
\begin{equation*}
\left[ Df\left( B\right) \right] :=\left\{ \left( a_{ij}\right)
_{i,j=1}^{n}:a_{ij}\in\left[ \inf_{p\in B}\frac{\partial f_{i}}{\partial
x_{j}}\left( p\right) ,\sup_{p\in B}\frac{\partial f_{i}}{\partial x_{j}}%
\left( p\right) \right] \text{for }i,j=1,\ldots,n\right\} . 
\end{equation*}
We refer to $[ Df\left( B\right) ] $ as the interval enclosure of the derivative of $f$ on $B$,
and write $Id$ for the identity matrix.

For an interval matrix $\mathbf{A}$, i.e. a set $\mathbf{A}\subset$ 
\correction{comment 10}{$\mathbb{R}^{n \times n}$}, we will write%
\begin{equation*}
\left\Vert \mathbf{A}\right\Vert :=\sup\left\{ \left\Vert Ax\right\Vert
:\left\Vert x\right\Vert =1,A\in\mathbf{A}\right\} . 
\end{equation*}
We say that $\mathbf{A}$ is invertible if each $A\in \mathbf{A}$ is
invertible. We define $\mathbf{A}^{-1}:=\{A^{-1}:A\in \mathbf{A}\}\subset 
\mathbb{R}^{n}\times\mathbb{R}^{n}$.

The following lemma is a version of the mean value theorem, which is useful in a number of places throughout the
paper.

\begin{lemma}
\label{lem:Df-difference}Let $f:\mathbb{R}^{n}\rightarrow\mathbb{R}^{n}$ be $%
C^{1}$, let $B\subset\mathbb{R}^{n}$ be a cartesian product of closed
intervals in $\mathbb{R}^{n}$ and let $p_{1},p_{2}\in B$, then we can choose
an $n\times n$ matrix $A\in\left[ Df\left( B\right) \right] $ for which we
will have 
\begin{equation*}
f\left( p_{1}\right) -f\left( p_{2}\right) =A\left( p_{1}-p_{2}\right) . 
\end{equation*}
with
\correction{comment 11}{
\begin{equation*}
A=\int_{0}^{1}Df\left( p_{2}+t\left( p_{1}-p_{2}\right) \right) dt. 
\end{equation*}}
\end{lemma}


We use the following classical result.

\begin{theorem}
\label{th:interval-Newton}\cite{Al} (Interval Newton method) Let $f:\mathbb{R%
}^{n}\rightarrow\mathbb{R}^{n}$ be a $C^{1}$ function and $B$ be a cartesian
product of closed intervals in $\mathbb{R}^{n}$. If $[Df(B)]$ is invertible
and there exists an $x_{0}$ in $B$ such that%
\begin{equation*}
N(x_{0},B):=x_{0}-\left[ Df(B)\right] ^{-1}f(x_{0})\subset B, 
\end{equation*}
then there exists a unique point $x^{\ast}\in B$ such that $f(x^{\ast})=0.$
\end{theorem}

\begin{figure}[tbp]
\begin{center}
\includegraphics[height=5cm]{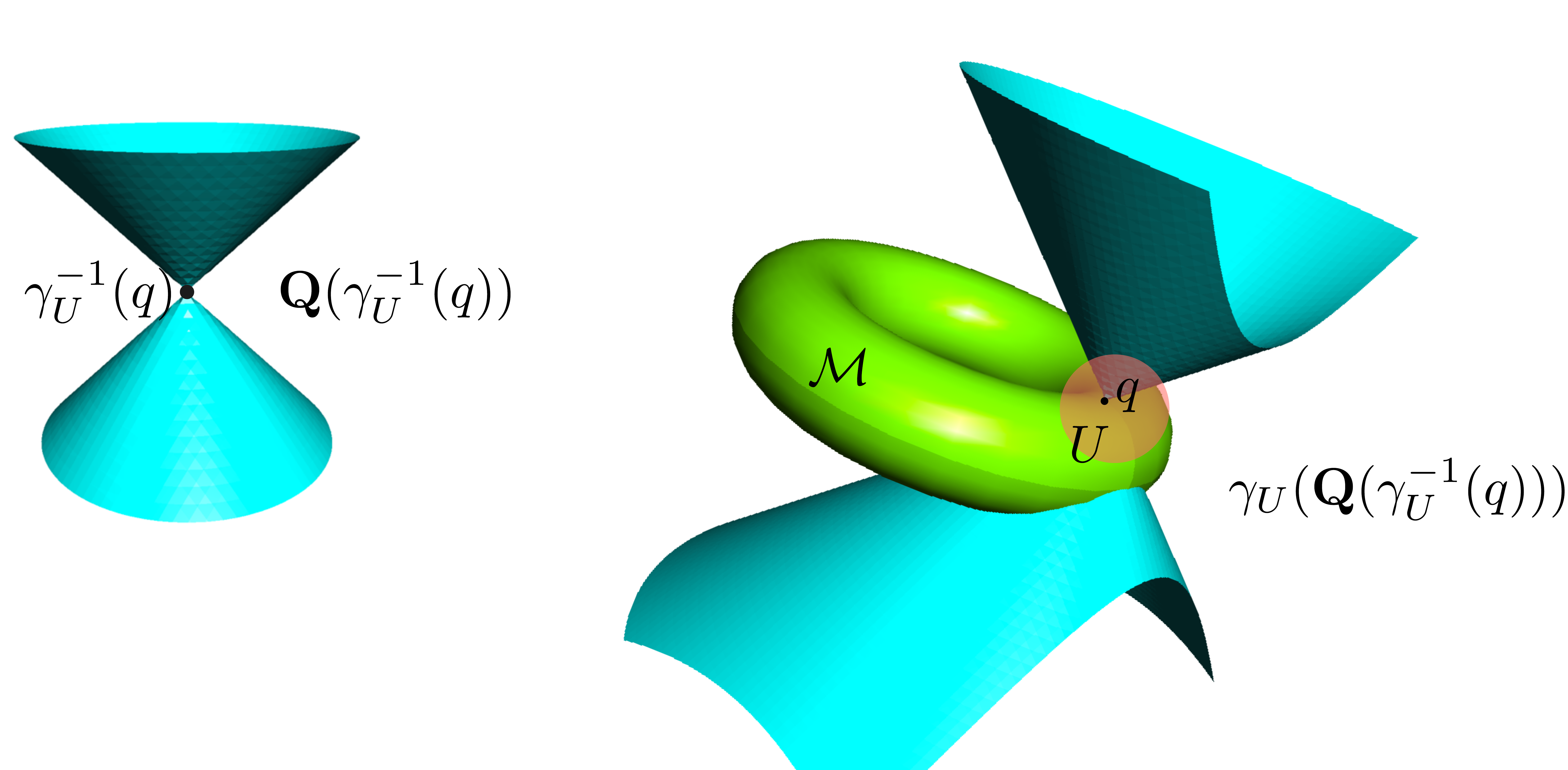}
\end{center}
\caption{On the left we have a cone attached at the point $\protect\gamma%
_U^{-1}(q)$ in the case when $k=2$ and $n=3$. Note that the cone is \emph{not%
} the blue (cone shaped) set. The cone $\mathbf{Q}(\protect\gamma_U^{-1}(q))$
is the complement of the blue set in $\mathbb{R}^3$; i.e. the white region
outside of the blue set. On the right we have an example of a Lipschitz
manifold. }
\label{fig:Lip-cone}
\end{figure}


The notion of a Lipschitz manifold requires us to define certain   
cone conditions. 
\correction{comment 23}{}\correction{comment 12}{}\correction{comment 22}{Fix $1 \leq k  \leq n$ and for 
a point $x=\left( x_{1},\ldots ,x_{n}\right) \in \mathbb{R}^{n}$ write 
$\pi_{x_1, \ldots, x_k}(x):=\left( x_1,\ldots ,x_{k}\right)$
and 
$\pi_{x_{k+1},\ldots ,x_{n}}(x) = (x_{k+1}, \ldots, x_n)$.
For a point $p\in \mathbb{R}^{n}$ we define
the cone attached to $p$ as (see Figure \ref{fig:Lip-cone})%
\begin{equation*}
\mathbf{Q}_{k} \left( p\right) :=\left\{ x\in \mathbb{R}^{n}:a\left\Vert 
\pi_{x_1, \ldots, x_k}\left( p-x\right) \right\Vert \geq \left\Vert \pi_{x_{k+1}, \ldots, x_n} \left(
p-x\right) \right\Vert \right\},
\end{equation*}
where $0<a \in \mathbb{R}$ is a fixed constant. We suppress the $k$ subscript and simply write $\mathbf{Q}(p)$ when 
$k$ is clear from context.
}

\begin{definition}
\label{def:lip-manifold} 
\correction{comment 24}{ Let $\mathcal{M} \subset \mathbb{R}^n$ be a $k$-dimensional
compact topological manifold. 
We say that $\mathcal{%
M}\subset \mathbb{R}^{n}$ is Lipschitz,
if it satisfies cone conditions in the following sense:}
 any point $p\in \mathcal{M}$
there exists an open neighborhood $U $ of $p$ in $\mathbb{R}^{n}$, an open
set $B\subset \mathbb{R}^{n}$, and a $C^{1}$ diffeomorphism $\gamma
_{U}:B\rightarrow U$ such that for any $q\in \mathcal{M} \cap U$ \correction{comment 21}{(see Figures \ref
{fig:Lip-cone}, \ref{fig:curve})}
\begin{equation*}
\mathcal{M}\cap U\subset \gamma _{U}\left( \mathbf{Q}\left( \gamma
_{U}^{-1}\left( q\right) \right) \cap B\right) .
\end{equation*}
\end{definition}


\section{Attracting invariant circles for maps on $\mathbb{R}^{2}$\label%
{sec:contractiing-maps}}

In this Section we discuss how to establish the existence of attracting invariant curves
for planar maps. The methodology is based on taking a neighborhood of the
curve and validating that this neighborhood maps into itself. This on its own ensures
only existence of an invariant set, and not that the set is a curve. We
therefore consider two additional conditions. The first is that we have a
`well aligned cone field' which also maps into itself, and the second is that we
have uniform contraction inside of the considered set.

The proposed method is similar in spirit to \cite{CZ0,CZ}. The main difference is that 
the papers just mentioned work with a vector bundle around the manifold. In the 
present work we formulate our results in local coordinates that roughly 
cover (see Figure \ref{fig:curve}) the investigated invariant curve,
removing the need for vector bundle coordinates. 
\correction{comment 13}{
This simplifies the implementation of the method.  
}

\correction{comment 16}{We give our proof in the setting of closed star-shaped invariant curves. Our results can be directly generalized to the setting where we haver a vector bundle  based on a closed curve (not necessarily star-shaped) in $\mathbb{R}^2$. We present this generalization in Section \ref{sec:generalization}. We give our proofs in the star-shaped setting due to the simplicity of the setup. The arguments for the more general case of vector bundles are analogous.}

In Section \ref{sec:Lip-curves} we present the
method which ensures the existence of Lipschitz invariant curves, in Section %
\ref{sec:Ck} we add conditions which ensure the $C^{k}$ smoothness, and in
Section \ref{sec:Lip-curves-validation} we discuss how the 
assumptions are validated in practice.

\subsection{Establishing closed attracting star-shaped curves\label%
{sec:Lip-curves}}

Let%
\begin{equation*}
f:\mathbb{R}^{2}\rightarrow\mathbb{R}^{2}
\end{equation*}
be a $C^{1}$ diffeomorphism\footnote{%
For the purposes of this Section we could assume that $f$ is a
homeomorphism, however the validation of the required assumptions is easier 
using the derivative of $f$.  This is why we assume $C^{1}$ smoothness
straightaway.}. Assume that $B_{1},B_{2}\subset\mathbb{R}^{2}$ are
homeomorphic to two dimensional open balls in $\mathbb{R}^{2}$ and that $%
\overline{B_{1}}\subset B_{2}$. Let $U:=\overline{B_{2}}\setminus B_{1}$ and
assume that 
\begin{equation}
U=\bigcup_{i=1}^{N}U_{i}=\bigcup_{i=1}^{N}\gamma_{i}\left( M_{i}\right) , \label{eq:U-set}
\end{equation}
\correction{comment 25}{
where  for each $i=1,\ldots,N$, the $M_{i}=\left[ -R_{i},R_{i}\right] \times\left[ -r_{i},r_{i}\right] $
for some fixed sequence of constants $0<r_{i},$ $R_{i}\in\mathbb{R}$, 
 and  $\gamma_i \colon M_i  \rightarrow \mathbb{R}^2$ are diffeomorphisms onto their image}
(See Figure \ref{fig:curve}). We
think of $\gamma_{i}$ as local coordinates on the set $U$.

\begin{figure}[ptb]
\begin{center}
\includegraphics[height=4.5cm]{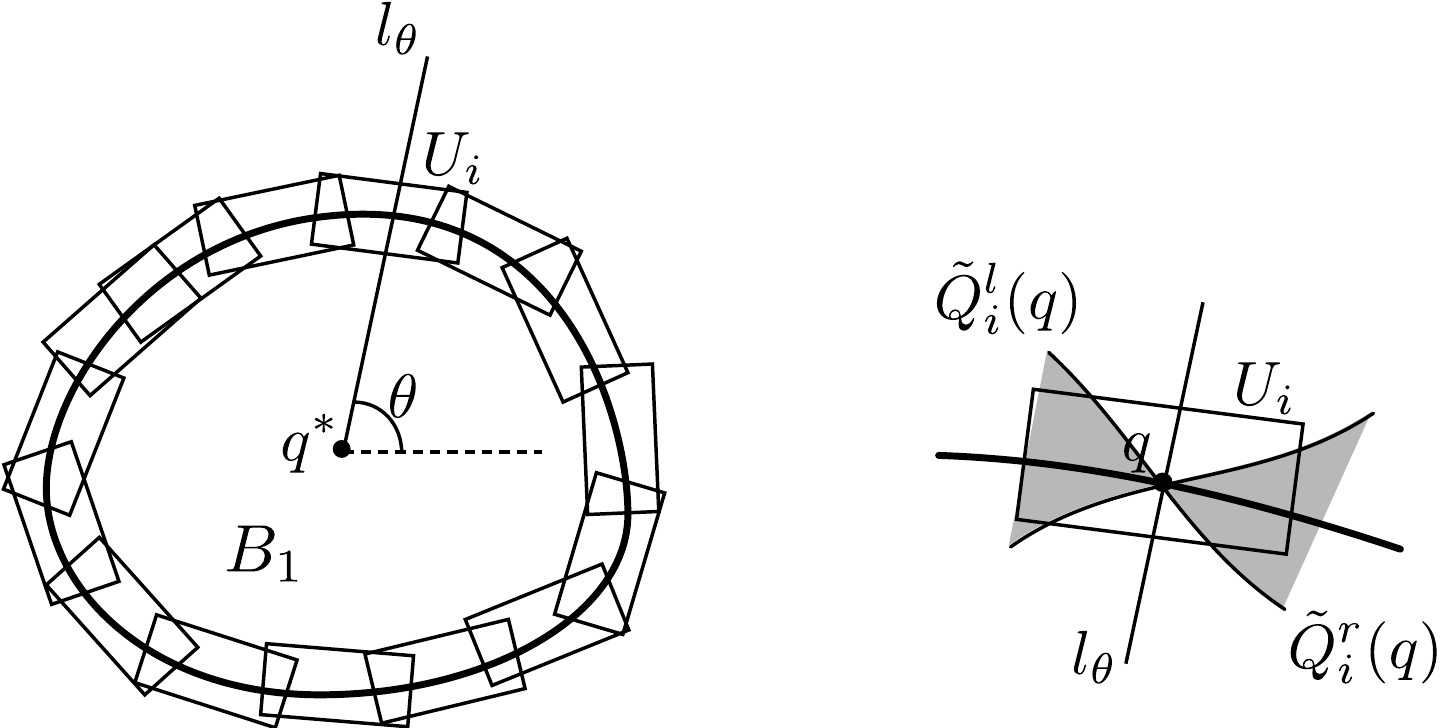}
\end{center}
\caption{The set $U$ is a collection of boxes, and we prove the 
existence of a star-shaped invariant closed
curve around $q^{\ast}$ which satisfies the cone conditions.}
\label{fig:curve}
\end{figure}

Our objective is to provide conditions ensuring the existence
of a star-shaped Lipschitz closed curve in $U$ homeomorphic to a circle and
invariant under $f$.

We equip each box $M_{i}$ with cones as follows. 
For $p\in\mathbb{R}^{2}$ define%
\begin{align}
Q_{i}\left( p\right) & =\left\{ \left( x,y\right) :\left\vert y-\pi
_{y}p\right\vert \leq a_{i}\left\vert x-\pi_{x}p\right\vert \right\} ,
\label{eq:cones-Lip-torus} \\
Q_{i}^{r}\left( p\right) & =Q_{i}\left( p\right) \cap\left\{ \left(
x,y\right) :x>\pi_{x}p\right\} ,  \notag \\
Q_{i}^{l}\left( p\right) & =Q_{i}\left( p\right) \cap\left\{ \left(
x,y\right) :x<\pi_{x}p\right\} ,  \notag
\end{align}
where $0<a_{1},\ldots,a_{N}\in\mathbb{R}$ are fixed constants. The
superscripts $r$ and $l$ stand for `right' and `left', respectively.
Define (see Figure \ref{fig:curve}) 
\begin{align*}
\tilde{Q}_{i}(q) & :=\gamma_{i}\left( Q_{i}\left( \gamma_{i}^{-1}(q)\right)
\right) , \\
\tilde{Q}_{i}^{\kappa}(q) & :=\gamma_{i}\left( Q_{i}^{\kappa}\left(
\gamma_{i}^{-1}(q)\right) \right) ,\qquad\text{for }\kappa\in\left\{
r,l\right\},
\end{align*}
and choose $q^{\ast}\in B_{1}\subset\mathbb{R}^{2}$.
(From now on the $q^{\ast}$ will remain fixed.) We define the half line
emanating from $q^{\ast}$ at an angle $\theta$ as%
\begin{equation*}
l_{\theta}:=\left\{ p\in\mathbb{R}^{2}:p=q^{\ast}+t\left( \cos\theta
,\sin\theta\right) \text{ for }t>0\right\} .
\end{equation*}

\begin{definition}\label{def:well-aligned-cones}
We say that the cones $Q_{i}$ are well aligned if for any $\theta\in
\lbrack0,2\pi),$ $i\in\left\{ 1,\ldots,N\right\} $ and $p\in\gamma_{i}^{-1}%
\left( U_{i}\cap l_{\theta}\right) $ we have%
\begin{equation*}
Q_{i}\left( p\right) \cap\gamma_{i}^{-1}(l_{\theta})=p,
\end{equation*}
and $\gamma_{i}^{-1}(l_{\theta})$ intersects $\left\{ y-\pi_{y}p=a_{i}\left(
x-\pi_{x}p\right) \right\} $ and $\left\{ y-\pi_{y}p=-a_{i}\left(
x-\pi_{x}p\right) \right\} $ transversally. (See Figure \ref%
{fig:well-aligned-cone}.)
\end{definition}

\begin{figure}[ptb]
\begin{center}
\includegraphics[height=2cm]{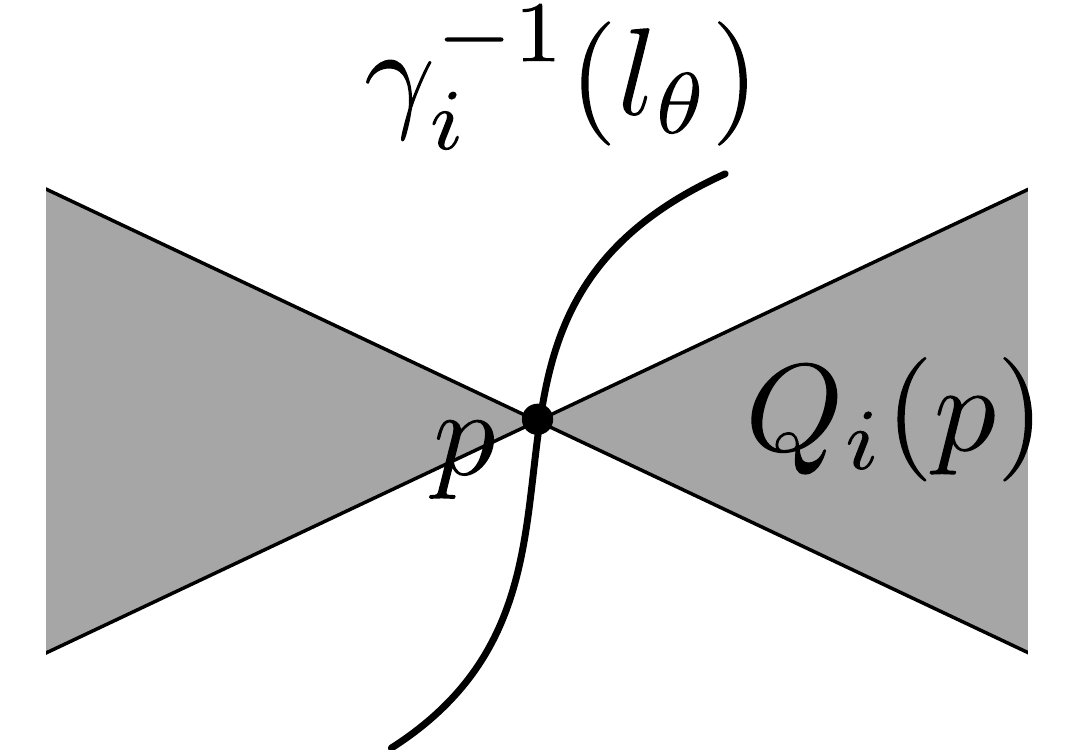}
\end{center}
\caption{A well aligned cone.}
\label{fig:well-aligned-cone}
\end{figure}

\begin{definition}
Let $h:\mathbb{S}^{1}\rightarrow\mathbb{R}^{2}$ be a continuous function. We
say that $h$ is a star-shaped closed curve around $q^{\ast}$ if %
\begin{equation*}
h\left( \mathbb{S}^{1}\right) \cap l_{\theta}=h\left( \theta\right),
\end{equation*}
\correction{comment 2}{for all $\theta \in \mathbb{S}^1$.}
\end{definition}

\begin{definition}\label{def:closed-curve-cone-cond}
We say that $h:\mathbb{S}^{1}\rightarrow\mathbb{R}^{2}$ is a
star-shaped closed curve
which satisfies cone conditions, if it is a closed curve around $q^{\ast}$,
and for any $\theta$ there exists an $i\in\left\{ 1,\ldots,N\right\} $ and $%
r>0$ such that (see Figure \ref{fig:curve})%
\begin{equation}
h\left( \mathbb{S}^{1}\right) \cap B\left( h\left( \theta\right) ,r\right)
\subset\tilde{Q}_{i}\left( h\left( \theta\right) \right) .  \label{eq:h-cc}
\end{equation}
\end{definition}

\begin{definition}
We say that $f$ satisfies cone conditions if for any $i\in\left\{
1,\ldots,N\right\} $ and any $p\in\mathrm{int}U_{i}$ there exists an $r>0$
and $j\in\left\{ 1,\ldots,N\right\} $ such that 
\begin{equation}
f\left( \tilde{Q}_{i}^{r}(p)\cap B(p,r)\right) \subset\tilde{Q}%
_{j}^{r}\left( f(p)\right) \quad\text{and\quad}f\left( \tilde{Q}%
_{i}^{l}(p)\cap B(p,r)\right) \subset\tilde{Q}_{j}^{l}\left( f(p)\right)
\label{eq:f-cc1}
\end{equation}
or%
\begin{equation}
f\left( \tilde{Q}_{i}^{l}(p)\cap B(p,r)\right) \subset\tilde{Q}%
_{j}^{r}\left( f(p)\right) \quad\text{and\quad}f\left( \tilde{Q}%
_{i}^{r}(p)\cap B(p,r)\right) \subset\tilde{Q}_{j}^{l}\left( f(p)\right) .
\label{eq:f-cc2}
\end{equation}
\end{definition}

\begin{definition}
Let $\mu$ be the Lebesgue measure on $\mathbb{R}^{2}$. We shall say that $f$
is uniformly attracting on $U$, if there exists a constant %
\correction{comment 26}{$0 \leq \lambda < 1$} such that for any Borel set $A\subset U$ we
have $\mu\left( f(A)\right)$ \correction{comment 26}{$\leq$} $\lambda\mu\left( A\right) $.
\end{definition}

We now formulate our main result with which we establish the existence
of attracting invariant curves.

\begin{theorem}
\label{th:lip-curve}Assume that the cones $Q_{i}$ are well aligned. Assume
also that there exists a sequence of points in $U$, such that the piecewise
affine circle which results from joining these points is a closed curve
which satisfies cone conditions around $q^{\ast}$. If $f$ is uniformly
contracting on $U$, and if $f$ satisfies cone conditions, and if $f\left( U\right)
\subset\mathrm{int}U$, then there exists a closed curve $h^{\ast}$ around $q^*$, which
satisfies cone conditions, such that $f\left( h^{\ast}\right) =h^{\ast}$.
\correction{comment 3}{ Moreover, for any $p\in U$, the orbit $\{f^{n}(p)\}_{n = 0}^\infty$ 
accumulates on the curve $h^{\ast}$.  That is, the $\omega$-limit set of 
the orbit is contained in $h^{\ast}$.}
\end{theorem}

\begin{proof}
The proof is based on the following graph transform type argument. Let $h$
be a closed curve around $q^{\ast}$, which satisfies cone conditions. 
We show that $f\left( h\left( \mathbb{S}^{1}\right)
\right) $ is the image of another closed curve around $q^{\ast} $, which
satisfies cone conditions. Then we show  iterates of $h$
converge to $h^*$.
Below we provide the details.

Since $f$ is a homeomorphism $f(h(\mathbb{S}^{1}))$ is a circle. We claim
that 
\begin{equation}
f(h(\mathbb{S}^{1}))\cap l_{\theta}\neq\emptyset\text{ \qquad for any }%
\theta\in\mathbb{S}^{1}.  \label{eq:proof-tmp-1}
\end{equation}
To see this, let $g:\mathbb{S}^{1}\rightarrow\mathbb{R}$ be defined as 
\begin{equation*}
g\left( \theta\right) =\left\{ 
\begin{array}{lll}
1 &  & \text{if }f(h(\mathbb{S}^{1}))\cap l_{\theta}\neq\emptyset, \\ 
0 &  & \text{otherwise.}%
\end{array}
\right.
\end{equation*}
Once we show that $g$ is continuous, this will prove
 (\ref{eq:proof-tmp-1}).  This is  because from $f(h(\mathbb{S}^{1}))\subset
f(U)\subset U$ and $U\cap q^{\ast}=\emptyset$ it follows that \correction{comment 27a}{$q^{\ast}\notin f(h(\mathbb{S}^{1}))$}
so for at least one $\theta\in\mathbb{S}^{1}$ we must have $%
g\left( \theta\right) =1$; then, by continuity, we will have $g\equiv1$. Since $g\equiv1$, this circle intersects $l_{\theta}$ for every $\theta \in \mathbb{S}^2$, which implies (\ref{eq:proof-tmp-1}).

To establish the continuity 
we start by showing that if $g\left( \theta\right) =0$, then for $\beta$
sufficiently close to $\theta$ we will have $g\left( \beta\right) =0.$
Suppose that $g\left( \theta\right) =0$. Since $f(h(\mathbb{S}^{1}))\subset
f(U)\subset U$ we see that $f(h(\mathbb{S}^{1}))$ and $l_{\theta}\cap U$ are
disjoint compact sets. This means that we can find their open neighborhoods,
which will also be disjoint. Therefore $f(h(\mathbb{S}^{1}))\cap l_{\beta
}=\emptyset$ for $\beta$ close to $\theta$, hence $g\left( \beta\right) =0$,
as required. 
\begin{figure}[ptb]
\begin{center}
\includegraphics[height=3.8cm]{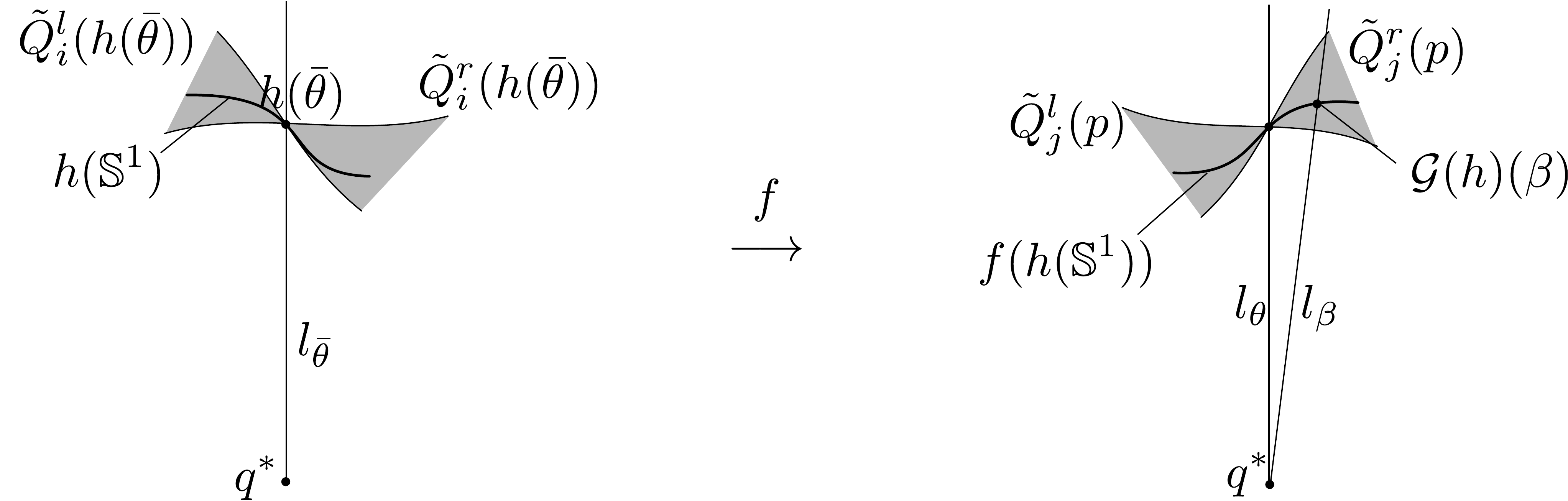}
\end{center}
\caption{Construction of $\mathcal{G}(h)$. }
\label{fig:graph-transform}
\end{figure}

We now show that if $g\left( \theta\right) =1$, then for $\beta$
sufficiently close to $\theta$ we will have $g\left( \beta\right) =1.$ Let $%
p\in f(h(\mathbb{S}^{1}))\cap l_{\theta}$. There existes a $\bar{\theta}$
such that $p=f\left( h\left( \bar{\theta}\right) \right) $. Take $i$ such
that $h\left( \bar{\theta}\right) \in U_{i}.$ From (\ref{eq:h-cc}), for $%
\bar{\beta}$ close to $\bar{\theta}$%
\begin{equation*}
h\left( \bar{\beta}\right) \in\tilde{Q}_{i}\left( h\left( \bar{\theta }%
\right) \right) .
\end{equation*}
Since $h$ is a closed curve around $q^{\ast}$ we see that we 
have either 
\begin{equation}
\bar{\beta}<\bar{\theta}\implies h\left( \bar{\beta}\right) \in\tilde{Q}%
_{i}^{l}\left( h\left( \bar{\theta}\right) \right) \quad\text{and\quad }\bar{%
\beta}>\bar{\theta}\implies h\left( \bar{\beta}\right) \in\tilde {Q}%
_{i}^{r}\left( h\left( \bar{\theta}\right) \right) ,
\label{eq:alignment-1-tmp}
\end{equation}
or%
\begin{equation}
\bar{\beta}<\bar{\theta}\implies h\left( \bar{\beta}\right) \in\tilde{Q}%
_{i}^{r}\left( h\left( \bar{\theta}\right) \right) \quad\text{and\quad }\bar{%
\beta}>\bar{\theta}\implies h\left( \bar{\beta}\right) \in\tilde {Q}%
_{i}^{l}\left( h\left( \bar{\theta}\right) \right) .
\label{eq:alignment-2-tmp}
\end{equation}
Without loss of generality let us assume that we have (\ref%
{eq:alignment-1-tmp}). (If we have (\ref{eq:alignment-2-tmp}) then the proof
follows from mirror arguments.) We know that $f$ satisfies cone conditions.
Let us therefore assume that we have (\ref{eq:f-cc1}), from which it follows
that \correction{comment 27b}{for some $j$ 
\begin{align}
f\left( h\left( \bar{\beta}\right) \right) & \in \tilde{Q}_{j}^{r}\left(
f(h\left( \bar{\theta}\right) )\right) =\tilde{Q}_{j}^{r}\left( p\right)
\qquad\text{for }\bar{\beta}>\bar{\theta},  \label{eq:cc-proof-1} \\
f\left( h\left( \bar{\beta}\right) \right) & \in \tilde{Q}_{j}^{l}\left(
f(h\left( \bar{\theta}\right) )\right) =\tilde{Q}_{j}^{l}\left( p\right)
\qquad\text{for }\bar{\beta}<\bar{\theta}.  \label{eq:cc-proof-2}
\end{align}}(If we have (\ref{eq:f-cc2}) then the proof will follow from mirror
arguments.) For any $\beta$ sufficiently close to $\theta$ there will
therefore exist a $\bar{\beta}$ such that $f\left( h\left( \bar{\beta }%
\right) \right) \in l_{\beta}$; see Figure \ref{fig:graph-transform}. Since $%
f\left( h\left( \bar{\beta}\right) \right) \in f\left( h\left( \mathbb{S}%
^{1}\right) \right) $ we see that for $\beta$ sufficiently close $g\left(
\beta\right) =1$, as required.

We have established (\ref{eq:proof-tmp-1}). \correction{comment 27c}{We will now define a function $%
\mathcal{G}(h):\mathbb{S}^{1}\rightarrow U$, which we will prove to be a closed curve around $q^*$ which has the same graph as $f(h(\mathbb{S}^1))$.
(We use the notation $\mathcal{G}(h)$ since the function follows from a
\textquotedblleft graph transform" type construction.) 
We start by taking $\theta=0$ and defining $\mathcal{G}(h)(0)$ to be any point from $f(h(\mathbb{S}^1))\cap l_{\theta}$. At this stage we do not know if such point is unique, so we choose an arbitrary point from the intersection. Take $j_0$ such that $\mathcal{G}(h)(0)\in U_{j_0}$. From (\ref{eq:cc-proof-1}--\ref{eq:cc-proof-2}) we see that  for $\theta>0$ close to zero, we can extend $\mathcal{G}(h)$ to obtain a curve by defining 
$\mathcal{G}(h)(\theta)=f(h(\mathbb{S}^1))\cap l_{\theta} \cap \tilde Q_{j_0} (\mathcal{G}(h)(0))$, 
as long as $\mathcal{G}(h)(\theta)$ remains in $U_{j_0}$. Let $\theta_1>0$ be some angle such that $\mathcal{G}(h)(\theta_1)\in U_{j_0}\cap U_{j_1}$, for some index $j_1$. We can then continue our construction for $\theta>\theta_1$ as 
$\mathcal{G}(h)(\theta)=f(h(\mathbb{S}^1))\cap l_{\theta} \cap \tilde Q_{j_1} (\mathcal{G}(h)(\theta_1))$. Continuing in this manner, we can reach $\theta=2\pi$. We are sure that at $\theta =2\pi$ we return to $\mathcal{G}(h)(0)$, since if this were not the case then we could continue with the construction and $f(h(\mathbb{S}^1))$ would contain an infinite spiral. This is not possible since $f(h(\mathbb{S}^1))$ is homeomorphic to a circe. 
 }From (\ref%
{eq:cc-proof-1}--\ref{eq:cc-proof-2}) we see that $\mathcal{G}(h)$ satisfies
cone conditions, which also implies that it is continuous.

We now show that by starting with the closed curve $h$ which connects
the points from the assumption of the theorem, then as we iterate the above
defined graph transform we shall converge to the curve we seek; i.e. 
\begin{equation*}
\lim_{n\rightarrow\infty}\mathcal{G}^{n}(h)=h^{\ast}.
\end{equation*}
Convergence follows from the assumption
that $f$ is uniformly contracting on $U$.

\begin{figure}[ptb]
\begin{center}
\includegraphics[height=4cm]{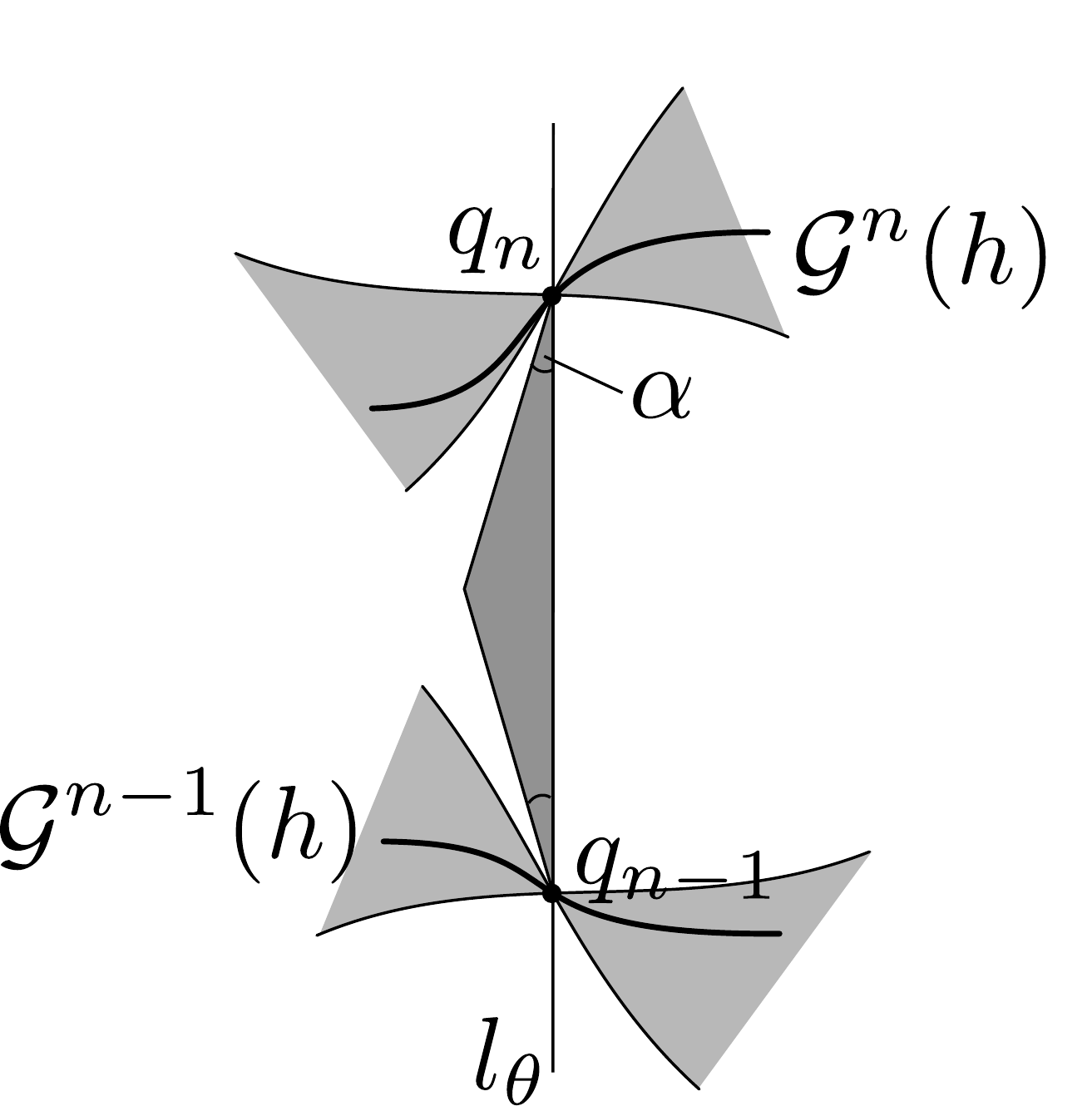}
\end{center}
\caption{Since $\mathcal{G}^{n}(h)$ and $\mathcal{G}^{n-1}(h)$ satisfy cone
conditions, we can find an angle $\protect\alpha$, such that the isosceles
triangle, with base joining $q_{n}$ and $q_{n-1}$, as in above plot, will
fit between $\mathcal{G}^{n}(h)$ and $\mathcal{G}^{n-1}(h)$. By compactness
of $U$ and the fact that we have a finite number of $C^{1}$ local maps $%
\protect\gamma_{i}$, the $\protect\alpha$ can be chosen independently of $n,%
\protect\theta,q_{n}$ and of $q_{n-1}$. This means that the area between $%
\mathcal{G}^{n}(h)$ and $\mathcal{G}^{n-1}(h)$ is bounded from below by $%
C\Vert q_{n}-q_{n-1}\Vert^{2}$, where $C>0$ is some constant independent
from $n$ and $\protect\theta$. }
\label{fig:area}
\end{figure}

For $n=1,2,\ldots $ let $A_{n}$ be the area between the curves $\mathcal{G}%
^{n}(h)$ and $\mathcal{G}^{n-1}(h)$. Since $f$ is uniformly contracting, $%
A_{n}\leq \lambda ^{n-1}A_{1}$. Let us consider two points, $q_{n}\in 
\mathcal{G}^{n}(h)\cap l_{\theta }$ and $q_{n-1}\in \mathcal{G}^{n-1}(h)\cap
l_{\theta }$, for some $\theta \in \mathbb{S}^{1}$. Since the curves $%
\mathcal{G}^{n}(h)$ and $\mathcal{G}^{n-1}(h)$ satisfy cone conditions the
area between them has to be at least $C\left\Vert q_{n}-q_{n-1}\right\Vert
^{2}$, where $C>0$ is a constant independent of the choice of $\theta $ and $%
n$. See Figure \ref{fig:area} and the caption below it. This gives us%
\begin{equation*}
C\left\Vert q_{n}-q_{n-1}\right\Vert ^{2}\leq A_{n}\leq \lambda ^{n-1}A_{1}
\end{equation*}%
so%
\begin{equation*}
\left\Vert q_{n}-q_{n-1}\right\Vert \leq \left( \sqrt{\lambda }\right) ^{n-1}%
\sqrt{\frac{A_{1}}{C}},
\end{equation*}%
from which, by the fact that $\sqrt{\lambda }<1$, it follows that the
sequence $q_{n}$ is convergent. This means that we can define \correction{comment 27d}{$h^{\ast
}\left( \theta \right) :=\lim_{n\to \infty}\mathcal{G}^{n}(h)\left( \theta \right) $}. All $%
\mathcal{G}^{n}(h)$ are closed curves around $q^{\ast }$, which satisfy cone
conditions. This property is preserved when passing to the limit, which
concludes the proof.
\end{proof}

\subsection{Smoothness\label{sec:Ck}}

In this Section we discuss how to establish that the invariant curve $%
h^{\ast }$ from Theorem \ref{th:lip-curve} is smooth. We first need to
introduce some notation.

Consider local maps $f_{ji}:\left[ -R_{i},R_{i}\right] \times \left[
-r_{i},r_{i}\right] \supset $domain$f_{ji}\rightarrow \mathbb{R}^{2}$
defined as%
\begin{equation*}
f_{ji}:=\gamma _{j}^{-1}\circ f\circ \gamma _{i}.
\end{equation*}%
(The domain of $f_{ji}$ can be empty.) We now define the following constants 
\begin{eqnarray*}
\xi &:=&\inf \left\{ \frac{\partial \left( \pi _{x}f_{ji}\right) }{\partial x%
}\left( q\right) -a_{j}\left\vert \frac{\partial \left( \pi
_{x}f_{ji}\right) }{\partial y}\left( q\right) \right\vert :i,j=1,\ldots
,N,~q\in \text{domain}f_{ij}\right\} , \\
\mu &:=&\sup \left\{ \left\vert \frac{\partial \left( \pi _{y}f_{ji}\right) 
}{\partial y}\left( q\right) \right\vert +a_{j}\left\vert \frac{\partial
\left( \pi _{x}f_{ji}\right) }{\partial y}\left( q\right) \right\vert
:i,j=1,\ldots ,N,~q\in \text{domain}f_{ij}\right\} .
\end{eqnarray*}

\begin{definition}
\label{def:rate-cond} We say that $f$ satisfies rate conditions of order $k$
if $\xi >0$ and for any $j\in \left\{ 1,\ldots ,k\right\} $%
\begin{equation}
\frac{\mu }{\xi ^{j+1}}<1.  \label{eq:rate-cond}
\end{equation}
\end{definition}

\begin{remark}
The definition is a simplified version of the rate condition considered in 
\cite{CZ}. There are two differences. The first is that in \cite{CZ} three
coordinates are considered: the unstable, the stable and central coordinate.
Here we only have two: the central coordinate $x$ and the stable coordinate $%
y$. The second difference is that the rate conditions considered in \cite{CZ}
include also bounds needed to establish the existence of the invariant
manifold. Here we do this using the slightly modified method from Theorem %
\ref{th:lip-curve}. Because of these two differences, the nine inequalities
for rate conditions from \cite{CZ} are reduced to the single inequality (\ref%
{eq:rate-cond}). (The condition (\ref{eq:rate-cond}) corresponds to the first
inequality in equation (4) from \cite{CZ}.)
\end{remark}
Below theorem is a reformulation of the smoothness result from \cite{CZ}, adapted to our simplified setting.

\begin{theorem}
\label{th:Ck}If in addition to all assumptions of Theorem \ref{th:lip-curve}
the map $f$ satisfies rate conditions of order $k,$ then $h^{\ast }$
established in Theorem \ref{th:lip-curve} is $C^{k}$.
\end{theorem}

\begin{proof}
The result follows from Lemma 48 in \cite{CZ}. In our setting the only
needed condition to apply Lemma 48 from \cite{CZ} is the condition (\ref%
{eq:rate-cond}); see the remark in the third bullet list item on page 6226
in \cite{CZ}. Lemma 48 from \cite{CZ} ensures that when we iterate the graph
transform from the proof of Theorem \ref{th:lip-curve} starting 
with a $C^{k}$ curve $h^{0},$ then the $C^{k}$ smoothness persists as we
pass to the limit. In the proof of Theorem \ref{th:lip-curve} we take $h^{0}$
to be a piecewise affine circle. We can smooth out the corners of such curve
to make it $C^{k}$, so the fact that the smoothness is preserved in the
limit ensures the claim.
\end{proof}

\begin{figure}[ptb]
\begin{center}
\includegraphics[height=4cm]{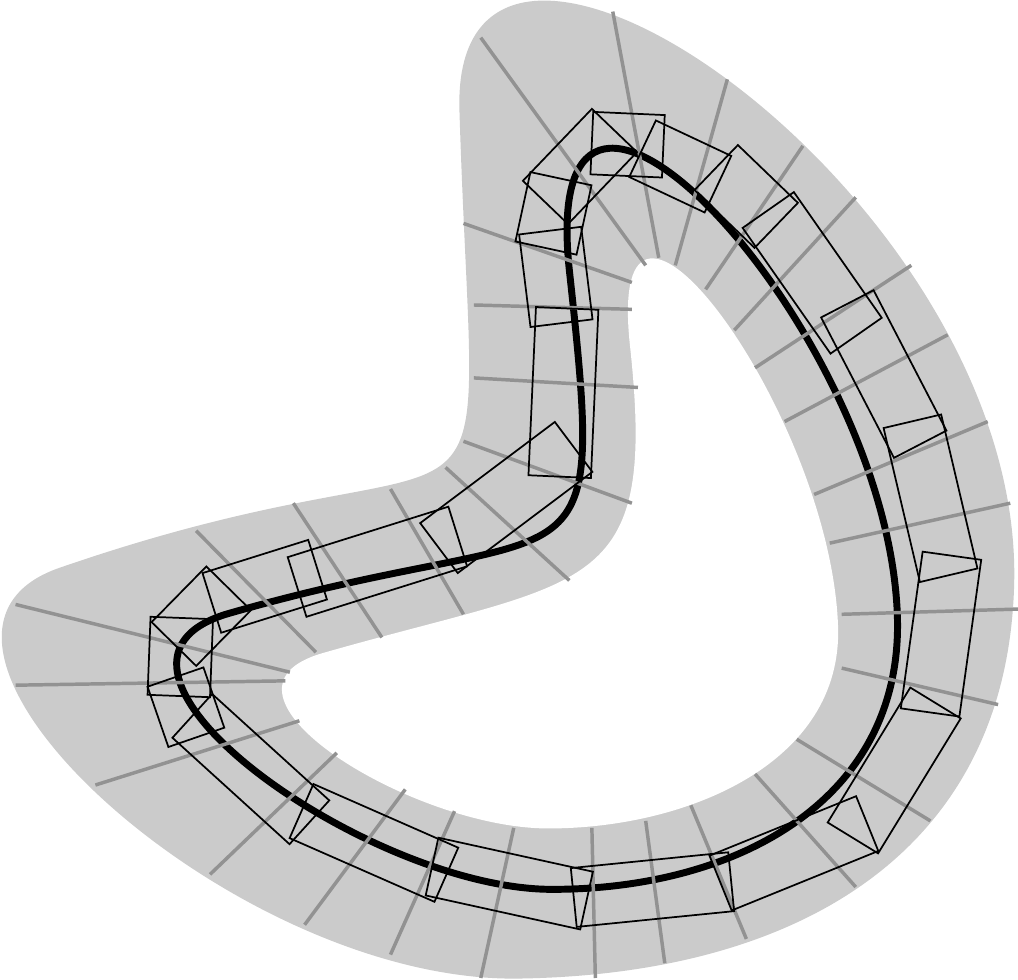}
\end{center}
\caption{Generalization to a vector bundle setting. The vector bundle $E$ is in grey, its base is the curve $p^*$, which is in black, with the fibers $E_{\theta}$ represented as the grey lines. The set $U$ consists of the union of the small rectangles. Note that in this picture $p^*$ is not the invariant curve, rather it is the base of the vector bundle.}
\label{fig:vector-bundle}
\end{figure}

\correction{comment 16}{\subsection{Generalization to the setting of vector bundles}\label{sec:generalization}
In this section we present how the results from Sections \ref{sec:Lip-curves} and \ref{sec:Ck} can be generalized. Let us
consider a closed curve $p^{\ast}\subset\mathbb{R}^{2}$, parameterized by
$\theta\in\mathbb{S}^{1}$, i.e. $p^{\ast}:\mathbb{S}^{1}\rightarrow
\mathbb{R}^{2}$. Consider a vector bundle $E$ in $\mathbb{R}^{2}$ with
$p^{\ast}(\mathbb{S}^{1})$ as its base and with fibers $E_{\theta}$ at
$p^{\ast}\left(  \theta\right)  $, for $\theta\in\mathbb{S}^{1}$. Assume that
the set $U\subset\mathbb{R}^{2}$ of the form (\ref{eq:U-set}) is a subset of
$E$ (see Figure \ref{fig:vector-bundle}). We will say that the family of cones $Q_{i}$ introduced
in the Section \ref{sec:Lip-curves} is well aligned with $E_{\theta}$, if it satisfies the
assumptions of Definition \ref{def:well-aligned-cones}, with $l_{\theta}$
changed to $E_{\theta}$. We will say that $h:\mathbb{S}^{1}%
\rightarrow\mathbb{R}^{2}$ is a closed curve around $p^{\ast}$ iff $h\left(
\mathbb{S}^{1}\right)  \cap E_{\theta}=h\left(  \theta\right)  $ for all $\theta\in \mathbb{S}^1$. With such
modiffications Theorems \ref{th:lip-curve} and \ref{th:Ck} remain true. Their
proofs in this more general setting are identical, with the only difference
that $l_{\theta}$ needs to be changed to $E_{\theta}$ and $q^{\ast}$ needs to
be changed to $p^{\ast}$ throughout the arguments.}

\correction{}{The main difficulty in this setting is actually constructing the needed vector bundles
in particular examples, 
and this is the technicality overcome using the star-shaped assumption in our 
earlier arguments. 
The interested reader is referred to \cite{CZ0,CZ} for more general discussion. }

\subsection{Validation of assumptions\label{sec:Lip-curves-validation}}

We finish this Section by describing how the assumptions of Theorems \ref%
{th:lip-curve}, \ref{th:Ck} are validated in practice. For the applications we
have in mind we take $\gamma _{i}$ to be affine maps, so checking that $Q_{i}$ are
well aligned is a simple linear algebra exercise of checking that $%
Q_{i}\left( p\right) $ and $\gamma _{i}^{-1}(l_{\theta })\ $ intersect
at one and only one point. For instance, when $\gamma _{i}\left( q\right)
=A_{i}q+q_{i}$ (with a matrix $A_{i}$ and point $q_{i}\in \mathbb{R}^{2}$)
one checks that for any $\theta $ such that $U_{i}\cap l_{\theta }\neq
\emptyset $ for $v_{\theta }=A_{i}^{-1}\left( \left( \cos \theta ,\sin
\theta \right) \right) $ we have $\left\vert \pi _{y}v_{\theta }\right\vert
>a_{i}\left\vert \pi _{x}v_{\theta }\right\vert $. This condition is 
checked with computer assistance.

Next we take a sequence of points $q_{1}\ldots,q_{N}$ around $q^{\ast}$ (in
our application we take the points $q_i$ to be the same as those used to
define the affine maps $\gamma_{i}$) and validate that lines joining $q_{i}$
with $q_{i+1}$ (and the line joining $q_{N}$ with $q_{1}$) lie inside the
cones $\tilde{Q}_{i}(q_{i})$ and $\tilde{Q}_{i+1}(q_{i+1})$, respectively
(and the cones $\tilde{Q}_{N}\left( q_{N}\right) $ and $\tilde{Q}_{1}(q_{1})$%
, respectively). This is also done with computer assistance.

To check that $f$ is uniformly contracting in $U$ we chevk that for
any $i\in \{1,\ldots,N\}$ and any matrix $A\in\left[ Df\left( U_{i}\right) %
\right] $ we have $\left\vert \det\left( A\right) \right\vert <1.$ This is
particularly simple to do in our case, as the computation of determinants of $%
2\times2$ matrices is straightforward. Once again, this is done with
computer assistance.

The condition that $f\left( U\right) $ is a subset of $U$ can be done
directly by taking $U_{i}=\bigcup_{k=1}^{m}U_{i,k}$ for some chosen $U_{i,k}$
and checking that for any $i\in \left\{ 1,\ldots ,N\right\} $ and $k\in
\left\{ 1,\ldots ,m\right\} $ there exists a $j\in \left\{ 1,\ldots
,N\right\} $ such that $f\left( U_{i,k}\right) \subset U_{j}$. In our
computer assisted approach we do this by taking 
\begin{equation*}
M_{i,k}:=\left[ -R_{i}+(k-1)\frac{2R_{i}}{m},-R_{i}+k\frac{2R_{i}}{m}\right]
\times \left[ -r_{i},r_{i}\right] ,
\end{equation*}%
$U_{i,k}:=\gamma _{i}\left( M_{i,k}\right) $ and checking that for some $%
j\in \left\{ 1,\ldots ,N\right\} $ 
\begin{equation}
\gamma _{j}^{-1}\circ f\circ \gamma _{i}\left( M_{i,k}\right) \subset M_{j}.
\label{eq:fragment-contraction}
\end{equation}%

The computation just described, which is done in the \textquotedblleft local
coordinates\textquotedblright , is well suited to interval arithmetic
implementation provided the maps $\gamma _{i}$ are well aligned with the
invariant curve we wish to established. When $\gamma _{l}$ are affine maps,
i.e. $\gamma _{l}\left( q\right) =A_{l}q+q_{l}$ condition (\ref%
{eq:fragment-contraction}) is validated by using Lemma \ref%
{lem:Df-difference} by taking a point $v_{i,k}\in M_{i,k}$ and checking that %
\begin{equation}
\gamma _{j}^{-1}\circ f\circ \gamma _{i}\left( M_{i,k}\right) \subset \gamma
_{j}^{-1}\circ f\circ \gamma _{i}\left( v_{i,k}\right) +\left( A_{j}^{-1} 
\left[ Df\left( \gamma _{i}\left( M_{i,k}\right) \right) \right]
A_{i}\right) \left( M_{i,k}-v_{i,k}\right) .
\label{eq:contraction-lip-curve-validation}
\end{equation}%
With a good choice of the matrices $A_{l}$ the matrix $A_{j}^{-1}\left[
Df\left( \gamma _{i}\left( M_{i,k}\right) \right) \right] A_{i}$ can be made
close to diagonal, which helps to reduce the wrapping effect of interval
arithmetic computations. In this case the derivatives of the local maps are
bounded by%
\begin{equation}
Df_{ji}(q)\in A_{j}^{-1}\left[ Df\left( \gamma _{i}\left( M_{i,k}\right)
\right) \right] A_{i}\qquad \text{for }q\in M_{i,k},
\label{eq:local-map-derivative}
\end{equation}%
and these can be used to compute the coefficients $\mu ,\xi $ needed for the
rate conditions.

The next lemma is used to validate cone conditions. We express it in a
more general setting where the map is defined on $\mathbb{R}^{n}$,
as this setting is needed in Section \ref{sec:homoclinic-maps}.
Below we state the result using the notations $f$ and $M$, but for our
purposes here one would apply it for a local map $f_{ji}$ on a set $M_{i,k}$,
with the bound on the derivative from (\ref%
{eq:local-map-derivative}). (We remove the subscripts to simplify
the statement and to make it more compatible with the story from Section \ref%
{sec:homoclinic-maps}.)
\begin{lemma}
\label{lem:cone-cond}Let $f:\mathbb{R}\times \mathbb{R}^{n-1}\rightarrow 
\mathbb{R}\times \mathbb{R}^{n-1}$ and let $M\subset \mathbb{R}^{n}$. Let $a_{1},a_{2}>0$ and 
\begin{eqnarray*}
Q_{1}\left( p\right) &=&\left\{ \left( x,y\right) \in \mathbb{R}\times 
\mathbb{R}^{n-1}:a_{1}\left\vert x-\pi _{x}p\right\vert \geq \left\Vert
y-\pi _{y}p\right\Vert \right\} , \\
Q_{2}\left( p\right) &=&\left\{ \left( x,y\right) \in \mathbb{R}\times 
\mathbb{R}^{n-1}:a_{2}\left\vert x-\pi _{x}p\right\vert \geq \left\Vert
y-\pi _{y}p\right\Vert \right\} .
\end{eqnarray*}%
\correction{comment 28}{}If for any $v=\left( 1,v_{y}\right) \in Q_{1}\left( 0\right) $ and any $A\in %
\left[ Df\left( M\right) \right] $ we have $Av\in Q_{2}\left( 0\right) $
then for any $p_{1},p_{2}\in M$ such that $p_{2}\in Q_{1}\left( p_{1}\right) 
$ we have $f\left( p_{2}\right) \in Q_{2}\left( f(p_{1})\right) $.
\end{lemma}

\begin{proof}
If $p_{1}=p_{2}$ the result is automatic. Assume that $p_{1}\neq p_{2}$.
Since $p_{2}\in Q_{1}\left( p_{1}\right) $ we see that $\pi _{x}\left(
p_{1}-p_{2}\right) \neq 0$. By Lemma \ref{lem:Df-difference} for some $A\in %
\left[ Df\left( B\right) \right] $ we have $f\left( p_{1}\right) -f\left(
p_{2}\right) =A\left( p_{1}-p_{2}\right) $. Take $v=\frac{p_{1}-p_{2}}{%
\left\vert \pi _{x}\left( p_{1}-p_{2}\right) \right\vert }$. Then since $%
Av\in Q_{2}\left( 0\right) $ we have $\left\Vert \pi _{y}Av\right\Vert \leq
a_{2}\left\vert \pi _{x}Av\right\vert $ so in turn \correction{comment 29}{}
\begin{align*} 
\left\Vert \pi _{y}(f\left( p_{1}\right) -f\left( p_{2}\right) )\right\Vert
& =\left\vert \pi _{x}\left( p_{1}-p_{2}\right) \right\vert \left\Vert \pi
_{y}Av\right\Vert  \\
& \leq \left\vert \pi _{x}\left( p_{1}-p_{2}\right) \right\vert
a_{2}\left\vert \pi _{x}Av\right\vert \\
& =a_{2}\left\vert \pi _{x}\left(
f\left( p_{1}\right) -f\left( p_{2}\right) \right) \right\vert ,
\end{align*}%
as required.
\end{proof}

\section{Heteroclinic connections between fixed points of maps\label%
{sec:homoclinic-maps}}

In this Section we discuss how to prove the existence of
heteroclinic orbits between two
fixed points of a map. We are interested in the case when one of the
fixed points is hyperbolic, and the other fixed point is a stable focus. The
heteroclinic orbits will be found in three steps. The first is to establish an
attracting neighborhood for which all trajectories converge to the
stable focus. This is discussed in Section \ref{sec:focus}. The second step
is to establish a bound on the unstable manifold of the hyperbolic fixed
point. This is described in Section \ref{sec:Wu}. Lastly we propagate the
unstable manifold by our map. If it reaches the attracting neighborhood of
the stable focus then we have established a heteroclinic orbit. This is
discussed in Section \ref{sec:heteroclinic-subsection}.

\subsection{Establishing attracting fixed points\label{sec:focus}}

In this Section we show how one can obtain the existence of an attracting
fixed point within a prescribed neighborhood. We start with a technical
lemma.
\begin{lemma}
\label{lem:contraction}Let $f:\mathbb{R}^{n}\rightarrow\mathbb{R}^{n}$ be $%
C^{1}$, and $\lambda>0$ be a fixed constant. Let $B$ be a cartesian product
of closed intervals in $\mathbb{R}^{n}$ (an $n$-dimensional cube). If for
any $A\in\left[ Df\left( B\right) \right] $ the matrix $\lambda Id-A^{\top}A$
is strictly positive definite, then for any $p_{1},p_{2}\in B$ 
\begin{equation*}
\left\Vert f\left( p_{1}\right) -f\left( p_{2}\right) \right\Vert
^{2}<\lambda\left\Vert p_{1}-p_{2}\right\Vert ^{2}.
\end{equation*}
\end{lemma}

\begin{proof}
By Lemma \ref{lem:Df-difference} we can choose an $A\in\left[ Df\left(
B\right) \right] $ such that $f\left( p_{1}\right) -f\left( p_{2}\right)
=A\left( p_{1}-p_{2}\right) $. (The choice of $A$ depends on $p_{1}$ and $%
p_{2}$.) We therefore have%
\begin{align*}
& \lambda\left\Vert p_{1}-p_{2}\right\Vert ^{2}-\left\Vert f\left(
p_{1}\right) -f\left( p_{2}\right) \right\Vert ^{2} \\
& =\lambda\left( p_{1}-p_{2}\right) ^{\top}\left( p_{1}-p_{2}\right) -\left(
f\left( p_{1}\right) -f\left( p_{2}\right) \right) ^{\top }\left( f\left(
p_{1}\right) -f\left( p_{2}\right) \right) \\
& =\lambda\left( p_{1}-p_{2}\right) ^{\top}\left( p_{1}-p_{2}\right) -\left(
p_{1}-p_{2}\right) ^{\top}A^{\top}A\left( p_{1}-p_{2}\right) \\
& =\left( p_{1}-p_{2}\right) ^{\top}\left( \lambda Id-A^{\top}A\right)
\left( p_{1}-p_{2}\right) \\
& >0,
\end{align*}
where the last line follows from the fact that $\lambda Id-A^{\top}A$ is
strictly positive definite.
\end{proof}

\begin{remark}
To check that a matrix $2\times2$ matrix $C$ is strictly positive definite,
it is enough to establish that%
\begin{equation*}
\det\left( C\right) >0\qquad\text{and\qquad trace}\left( C\right) >0.
\end{equation*}
\end{remark}

The next lemma establishes that we have an attracting fixed point within a
prescribed neighborhood.
\begin{lemma}
\label{lem:contraction-main}Let $f:\mathbb{R}^{n}\rightarrow\mathbb{R}^{n}$
be $C^{1}$. Let $\lambda\in\left( 0,1\right) $ be a fixed constant. Let $B$
be a cartesian product of closed intervals in $\mathbb{R}^{n}$ (an $n$%
-dimensional cube). If $f\left( B\right) \subset B$ and for any $A\in\left[
Df\left( B\right) \right] $ the matrix $\lambda Id-A^{\top}A$ is strictly
positive definite, then there exists an attracting fixed point of $f$ in $B$%
. (By \textquotedblleft attracting" we mean that for any $p\in B$, $f^{k}(p)
$ will converge to the fixed point as $k$ tends to infinity.)
\end{lemma}

\begin{proof} \correction{comment 30}{Since $\lambda\in (0,1)$, by Lemma \ref{lem:contraction} we see that $f$ is contracting, so the result follows from the Banach fixed point theorem.}
\end{proof}

\subsection{Establishing unstable manifolds of hyperbolic fixed points\label%
{sec:Wu}}

We now give a method for establishing mathematically rigorous bounds 
for a local unstable manifold of a hyperbolic fixed point. We restrict to the case where
the unstable manifold is of dimension $1$ as this is 
the case seen in the applications. Our method is 
based on \cite{Cap-Lyap}, and a more general procedure is found
in \cite{Zgliczynski-cone-cond}.

Let $p^{\ast}$ be a hyperbolic fixed point of a $C^{1}$ map $f:\mathbb{R}%
^{n}\rightarrow\mathbb{R}^{n}$. Assume that the unstable eigenspace of $%
p^{\ast}$ is of dimension $u=1$. Assume that the unstable eigenvalue of $%
Df(p^{\ast})$ is $\lambda$, with $|\lambda|>1$.

Let $B_{u}$ be a closed interval and let $B_{s}$ be an $s:=n-u=n-1$
dimensional product of closed intervals (a closed cube in $\mathbb{R}^{s}$).
Let $B=B_{u}\times B_{s}$ and assume that $p^{\ast}\subset\mathrm{int}B.$
For any point $p\in\mathbb{R}^{n}=\mathbb{R}^{u}\times\mathbb{R}^{s}$ we
shall write $p=\left( p_{u},p_{s}\right) $. The subscripts $u$ and $s$ stand
of \textquotedblleft unstable" and \textquotedblleft stable", respectively.
This notation is chosen since in our approach these coordinates will be
roughly aligned with the unstable/stable eigenspaces of $p^{\ast}$. We will
use the notation $\pi_{u}p=p_{u}$ and $\pi_{s}p=p_{s}$ for the projections.

Let $L>0$ be a fixed constant. For any $p=(p_u,p_s)\in\mathbb{R}^{u}\times%
\mathbb{R}^{s}$ we define a cone centered at $p$ as%
\begin{equation}
Q\left( p\right) :=\left\{ q=\left( q_{u},q_{s}\right) :\left\Vert
q_{s}-p_{s}\right\Vert \leq L\left\vert q_{u}-p_{u}\right\vert \right\} .
\label{eq:cone-def}
\end{equation}

\begin{definition}
We say that $h:B_{u}\rightarrow B_{u}\times B_{s}$ is a horizontal disc in $%
B $ if it is continuous, if for any $x\in B_{u}$, $\pi_{u}h\left( x\right) =x$
and if $h\left( B_{u}\right) \subset Q\left( h(x)\right) $.
\end{definition}

In other words, horizontal discs are one dimensional curves in $B_{u}\times
B_{s}$, which are graphs of Lipschitz functions with the Lipschitz constant
$L.$

The next lemma is our main tool for establishing bounds on the unstable
manifold of $p^*$.

\begin{lemma}
\label{lem:wu}\cite{Cap-Lyap}Assume that for any $p_{1},p_{2}\in B$ such
that $p_{2}\in Q\left( p_{1}\right) $ we have 
\begin{equation}
f(p_{2})\in Q\left( f\left( p_{1}\right) \right) .  \label{eq:cone-cond-hyp}
\end{equation}
Let $m\in\left( 1,\left\vert \lambda\right\vert \right) $ be a fixed number.
Assume that for any $p\in(Q(p^{\ast})\cap B)\setminus\{p^{\ast}\}$ we have%
\begin{equation}
\left\Vert f\left( p\right) -p^{\ast}\right\Vert >m\left\Vert p-p^{\ast
}\right\Vert .  \label{eq:expansion-hyp}
\end{equation}
Then the unstable manifold of $p^{\ast}$ is contained in $Q\left( p^{\ast
}\right) $. Moreover, there exists a horizontal disc $w^{u}:B_{u}\rightarrow
B$ in $B$ such that the unstable manifold of $p^{\ast}$ is the graph of $%
w^{u}.$
\end{lemma}

We now discuss validation of the assumptions (\ref{eq:cone-cond-hyp}), (%
\ref{eq:expansion-hyp}) of Lemma \ref{lem:wu}. To verify (\ref%
{eq:cone-cond-hyp}) we  use Lemma \ref{lem:cone-cond} (taking $M=B$ and $%
a_{1}=a_{2}=L$). To verify (\ref{eq:expansion-hyp}) we use the following
lemma.

\begin{lemma}
\cite{Cap-Lyap}Assume that%
\begin{equation*}
\left[ Df\left( B\right) \right] =\left( 
\begin{array}{ll}
\mathbf{A} & \mathbf{C}_{12} \\ 
\mathbf{C}_{21} & \mathbf{B}%
\end{array}%
\right) ,
\end{equation*}%
where $\mathbf{A}=\left[ a_{1},a_{2}\right] $ is a closed interval, $\mathbf{%
B}$, $\mathbf{C}_{12}$ and $\mathbf{C}_{21}$ are $s\times s$, $1\times s$
and $s\times 1$ interval matrices, respectively. If 
\begin{equation*}
a_{1}-\left\Vert \mathbf{C}_{12}\right\Vert L>m\sqrt{1+L^{2}}
\end{equation*}%
then for any $p\in (Q(p^{\ast })\cap B)\setminus \{p^{\ast }\}$ we have (\ref%
{eq:expansion-hyp}).
\end{lemma}

\subsection{Establishing heteroclinic connections\label%
{sec:heteroclinic-subsection}}
In this Section we combine the results of Sections \ref{sec:focus}, \ref{sec:Wu} to
obtain a heteroclinic orbit between two fixed points of a map $f:\mathbb{R}%
^{n}\rightarrow\mathbb{R}^{n}$ in the special case that one of the fixed points 
is attracting.  The existence of the attracting fixed point is established using the tools from 
Section \ref{sec:focus}. The other fixed point is hyperbolic, and has a one dimensional 
unstable manifold, as in Section \ref{sec:Wu}. The next theorem is used to establish 
homoclinic connections between such two points.
\correction{comment 4}{Computer assisted methods of proof for more general 
configurations are discussed in }
\cite{MR3461310,MR3281845,MR3068557,MR3207723}.

\begin{theorem}
\label{th:connecting-curve} Let $f:\mathbb{R}^{n}\rightarrow\mathbb{R}^{n}$
be $C^{1}$ and $B^{1},B^{2}\subset\mathbb{R}^{n}$ be two sets which are
cartesian products of closed intervals in $\mathbb{R}^n$. Assume that 
the set $B^{1}$ satisfies the assumptions of Lemma \ref{lem:contraction-main}. 
(that is, the assumptions hold for $B=B^{1}$.)

Assume also that $p_{2}^{\ast }\in
B^{2}=B_{u}\times B_{s}$ 
is a hyperbolic fixed point and that the
assumptions of Lemma \ref{lem:wu} are
satisfied.

\begin{enumerate}
\item If there exists an $n\geq 0$ and $\bar{x}\in B_{u}$ such that $%
f^{n}(Q\left( p_{2}^{\ast }\right) \cap \left\{ p:\pi _{u}p=\bar{x}\right\}
)\subset B^{1}$, then there exists an attracting fixed point $p_{1}^{\ast
}\in B^{1}$ and a homoclinic orbit from $p_{1}^{\ast }$ to $p_{2}^{\ast }$.

\item If there exists an $n\geq 0$, an interval $I\subset B_{u}$, and an $%
\bar{x}\in I$ such that $\pi _{u}f\left( Q\left( p_{2}^{\ast }\right) \cap
\left\{ p:\pi _{u}p=\bar{x}\right\} \right) \subset I$ and $f^{n}(\left\{
p\in Q\left( p_{2}^{\ast }\right) :\pi _{u}p\in I\right\} )\subset B^{1}$,
then there exists a $C^{0}$ curve, invariant under $f$, which joins $%
p_{1}^{\ast }$ and $p_{2}^{\ast }$.
\end{enumerate}
\end{theorem}

\begin{proof}
We start by proving the first claim. We have $B_{u}\times B_{s}=B^{2}$ and
by Lemma \ref{lem:wu} there exists the function $w^{u}:B_{u}\rightarrow B^{2}
$ which parameterizes the unstable manifold of $p_{2}^{\ast }$. Since $w^{u}$
is a horizontal disc, it has the properties that $\pi _{u}w^{u}(x)=x$ and
that for any $x\in B_{u}$, $w^{u}\left( B_{u}\right) \subset Q\left(
w^{u}(x)\right) $; in particular $w^{u}\left( B_{u}\right) \subset Q\left(
p_{2}^{\ast }\right) $. This means that $w^{u}\left( \bar{x}\right) \in
Q\left( p_{2}^{\ast }\right) \cap \left\{ p:\pi _{u}p=\bar{x}\right\} $.
This by the assumption of our lemma implies that 
\begin{equation*}
f^{n}(w^{u}\left( \bar{x}\right) )\subset f^{n}\left( Q\left( p_{2}^{\ast
}\right) \cap \left\{ p:\pi _{u}p=\bar{x}\right\} \right) \subset B^{1}.
\end{equation*}
By Lemma \ref{lem:contraction-main} the point $p_{1}^{\ast }$ is attracting
in $B^{1}$, which means that \[\lim_{k\rightarrow +\infty}f^{k}(w^{u}\left( \bar{x}\right) )=p_{1}^{\ast }.\] Since $w^{u}$
parameterizes the unstable manifold of $p_{2}^{\ast }$ we also have \[
\lim_{k\rightarrow -\infty }f^{k}(w^{u}\left( \bar{x}\right) )=p_{2}^{\ast },\]
 which concludes the proof of the first claim.

To prove the second claim observe that $w^{u}\left( \bar{x}\right) \in
Q\left( p_{2}^{\ast }\right) \cap \left\{ p:\pi _{u}p=\bar{x}\right\} $, so 
\begin{equation*}
\pi _{u}f\left( w^{u}\left( \bar{x}\right) \right) \subset \pi _{u}f\left(
Q\left( p_{2}^{\ast }\right) \cap \left\{ p:\pi _{u}p=\bar{x}\right\}
\right) \subset I.
\end{equation*}%
Let $x_{1}:=\bar{x}$ and $x_{2}:=\pi _{u}f\left( w^{u}\left( \bar{x}\right)
\right) $. The curve $w^{u}\left( \left[ x_{1},x_{2}\right] \right) $ is a
fragment of the unstable manifold, which joins the point $w^{u}\left( \bar{x}%
\right) $ with the point $f\left( w^{u}\left( \bar{x}\right) \right) $. This
means that for any $N\in \mathbb{N}$ we can define a continuous curve 
\begin{equation*}
\gamma _{N}:=w^{u}\left( B_{u}\right) \cup
\bigcup_{k=1}^{N}f^{k}(w^{u}\left( \left[ x_{1},x_{2}\right] \right) ),
\end{equation*}%
which coincides with a fragment of the unstable manifold. (The larger the $N$
the larger the fragment). From our assumption%
\begin{equation*}
f^{n}(w^{u}\left( \left[ x_{1},x_{2}\right] \right) \subset f^{n}(\left\{
p\in Q\left( p_{2}^{\ast }\right) :\pi _{u}p\in I\right\} )\subset B^{1}.
\end{equation*}
Since $f$ is contracting on $B^{1}$, $f^{k}(w^{u}\left( \left[ x_{1},x_{2}%
\right] \right) $ converge to $p_{1}^{\ast }$ as $k$ tends to infinity. This
means that 
\begin{equation*}
\gamma :=\bigcup_{N=1}^{\infty }\gamma _{N}\cup \left\{ p_{1}^{\ast
}\right\} 
\end{equation*}%
is a continuous curve joining $p_{1}^{\ast }$ and $p_{2}^{\ast }$, as
required.
\end{proof}


\section{Attracting invariant tori of ODEs in $\mathbb{R}^{3}$\label%
{sec:tori-3d}}

Consider a $C^{l}$, $l\ge 1$ vector field $F:\mathbb{R}^{3}\rightarrow \mathbb{R}^{3}$.
We are interested in the dynamics of the ODE%
\begin{equation}
x^{\prime }=F\left( x\right) .  \label{eq:ode-3d}
\end{equation}%
Our goal is to establish two types of invariant tori for the
flow of (\ref{eq:ode-3d}). First, an attracting torus which is
either $C^{k}$ smooth, with $k\le l$, or Lipschitz. The second 
is a torus that results from homoclinic connections of stable/unstable 
manifolds of periodic orbits.

Both types of tori are established by considering a section 
$\Sigma \subset \mathbb{R}^{3}$ and a section to section map $P:\Sigma \rightarrow
\Sigma $ induced by the flow of the ODE. The first type of torus follows from the construction of  invariant
curves by taking $f=P$ and using the tools from Section \ref%
{sec:contractiing-maps}. The second type follows from homoclinic
connections between $m$-periodic orbits of $P$, which are established by
taking $f=P^{m}$ and using the methods of Section \ref{sec:homoclinic-maps}.

The following theorem ensures that the invariant
circle established using tools from Section \ref{sec:contractiing-maps}
leads to an invariant Lipschitz torus.
Let $\Phi _{t}$ denote the flow induced by (\ref{eq:ode-3d}).

\begin{theorem}
\label{th:Lip-torus} Assume that $h^{\ast }:\mathbb{S}^{1}\rightarrow \Sigma 
$ is a closed invariant curve (invariant for $f=P$). Let 
\begin{equation*}
\mathcal{T}:=\left\{ \Phi _{t}(v):v\in h^{\ast }(\mathbb{S}^{1}),t\in 
\mathbb{R}\right\} .
\end{equation*}

\begin{enumerate}
\item If $h^{\ast }$ satisfies cone conditions
 (in the sense of Definition \ref{def:closed-curve-cone-cond}) 
 then $\mathcal{T}$ is a (two
dimensional) Lipschitz invariant torus for (\ref{eq:ode-3d}). 

\item If $h^{\ast }$ is $C^{k}$ then $\mathcal{T}$ is a $C^{k}$ invariant
torus for (\ref{eq:ode-3d}).
\end{enumerate}
\end{theorem}

\begin{proof}
The set $\mathcal{T}$ is a torus by construction. So we  need to show that
it is Lipschitz in the sense of Definition \ref{def:lip-manifold}.

Take $p=\Phi _{t}\left( v\right) \in \mathcal{T}$. Assume that $v\in U_{i}$,
meaning that $v\in \Sigma $ is in the local coordinates given by $\gamma
_{i} $ on $\Sigma $. (Throughout the reminder of the proof the $p,v$ and $t$
shall remain fixed.) We extend $\gamma _{i}$ to a neighborhood of $p$
by defining 
\begin{equation*}
\tilde{\gamma}_{i}\left( x_{1},x_{2},x_{3}\right) :=\Phi _{t+x_{2}}\left(
\gamma _{i}\left( a_{i}x_{1},x_{3}\right) \right) .
\end{equation*}%
(Note that in $\tilde{\gamma}_{i}$ we have added a rescaling on the
coordinate $x_{1}$. The $a_i$ used for the rescaling are the parameters from 
the cones $Q_i$ in (\ref{eq:cones-Lip-torus}).) Take a small ball $B$ around $\tilde{\gamma}%
_{i}^{-1}\left( p\right) $ and define $U:=\tilde{\gamma}_{i}\left( B\right) $
and 
\begin{equation*}
\gamma _{U}:=\tilde{\gamma}_{i}|_{B}.
\end{equation*}
Above $U$ and $\gamma_U$ are those needed for Definition \ref{def:lip-manifold}.

We need that 
\begin{equation*}
U\cap \mathcal{T}\subset \gamma _{U}\left( \mathbf{Q}\left( \gamma
_{U}^{-1}\left( \bar{q}_{1}\right) \right) \cap B\right) .
\end{equation*}%
This is equivalent to the condition that for any $q_{1},q_{2}\in U\cap \mathcal{T}
$%
\begin{equation}
\left\Vert \pi _{x_{1},x_{2}}\left( \gamma _{U}^{-1}\left( q_{1}\right)
-\gamma _{U}^{-1}\left( q_{2}\right) \right) \right\Vert \geq \left\Vert \pi
_{x_{3}}\left( \gamma _{U}^{-1}\left( q_{1}\right) -\gamma _{U}^{-1}\left(
q_{2}\right) \right) \right\Vert , \label{eq:lip-temp-prf}
\end{equation}%
which is what we show below.

Before proving (\ref{eq:lip-temp-prf}) we make the following 
auxiliary observation. Consider first $\bar{q}_{1},\bar{q}_{2}\in \mathcal{T}\cap \Sigma $. (Here
we do not need $\bar{q}_{1},\bar{q}_{2}$ to be in $U$. In fact, if $p$ is
far from $\Sigma $ such $\bar{q}_{1},\bar{q}_{2}$ will not be in $U$.) Since 
$\bar{q}_{1},\bar{q}_{2}\in h^{\ast }(\mathbb{S}^{1})$ from the fact that $%
h^{\ast }$ satisfies cone conditions it follows that $q_{2}\notin \tilde{Q}%
_{i}\left( q_{1}\right) $ hence 
\begin{equation}
\bar{q}_{2}\notin \tilde{Q}_{i}\left( q_{1}\right) =\gamma _{i}\left(
Q_{i}\left( \gamma _{i}^{-1}\left( \bar{q}_{1}\right) \right) \right) =%
\tilde{\gamma}_{i}\left( \mathbf{Q}\left( \tilde{\gamma}_{i}^{-1}\left( \bar{%
q}_{1}\right) \right) \right) \cap \Sigma =\gamma _{U}\left( \mathbf{Q}%
\left( \gamma _{U}^{-1}\left( \bar{q}_{1}\right) \right) \right) \cap \Sigma
.  \label{eq:cone-bar-alignment}
\end{equation}%
Since $\bar{q}_{1},\bar{q}_{2}\in \Sigma $ 
\begin{equation*}
\pi _{x_{2}}\tilde{\gamma}_{i}^{-1}\left( \bar{q}_{1}\right) =\pi _{x_{2}}%
\tilde{\gamma}_{i}^{-1}\left( \bar{q}_{2}\right) =-t,
\end{equation*}%
so (\ref{eq:cone-bar-alignment}) implies%
\begin{equation}
\left\Vert \pi _{x_{1},x_{2}}\left( \gamma _{U}^{-1}\left( \bar{q}%
_{1}\right) -\gamma _{U}^{-1}\left( \bar{q}_{2}\right) \right) \right\Vert
\geq \left\Vert \pi _{x_{3}}\left( \gamma _{U}^{-1}\left( \bar{q}_{1}\right)
-\gamma _{U}^{-1}\left( \bar{q}_{2}\right) \right) \right\Vert .
\label{eq:Lip-on-Sigma}
\end{equation}

Then we are ready to show (\ref{eq:lip-temp-prf}).
Take $q_{1},q_{2}\in U\cap \mathcal{T}$ where $q_{1}=\Phi _{t_{1}}\left( 
\bar{q}_{1}\right) ,$ $q_{2}=\Phi _{t_{2}}\left( \bar{q}_{2}\right) \ $for $%
\bar{q}_{1},\bar{q}_{2}\in \mathcal{T}\cap \Sigma $ and some $t_{1},t_{2}\in 
\mathbb{R}$. By (\ref{eq:Lip-on-Sigma}) we obtain%
\begin{eqnarray*}
\left\Vert \pi _{x_{1},x_{2}}\left( \gamma _{U}^{-1}\left( q_{1}\right)
-\gamma _{U}^{-1}\left( q_{2}\right) \right) \right\Vert  &\geq &\left\Vert
\pi _{x_{1}}\left( \gamma _{U}^{-1}\left( q_{1}\right) -\gamma
_{U}^{-1}\left( q_{2}\right) \right) \right\Vert  \\
&=&\left\Vert \pi _{x_{1}}\left( \gamma _{U}^{-1}\left( \bar{q}_{1}\right)
-\gamma _{U}^{-1}\left( \bar{q}_{2}\right) \right) \right\Vert  \\
&=&\left\Vert \pi _{x_{1},x_{2}}\left( \gamma _{U}^{-1}\left( \bar{q}%
_{1}\right) -\gamma _{U}^{-1}\left( \bar{q}_{2}\right) \right) \right\Vert 
\\
&\geq &\left\Vert \pi _{x_{3}}\left( \gamma _{U}^{-1}\left( \bar{q}%
_{1}\right) -\gamma _{U}^{-1}\left( \bar{q}_{2}\right) \right) \right\Vert 
\\
&=&\left\Vert \pi _{x_{3}}\left( \gamma _{U}^{-1}\left( q_{1}\right) -\gamma
_{U}^{-1}\left( q_{2}\right) \right) \right\Vert ,
\end{eqnarray*}%
as required.

The second claim follows directly from the fact that $(t,x)\rightarrow \Phi
_{t}(x)$ is $C^{k}$.
\end{proof}


The next result ensures that a homoclinic connection established using
the tools from Section \ref{sec:homoclinic-maps}  gives an invariant
torus for the ODE.
\begin{lemma}\label{lem:resonant-torus}
Assume that $p_{1}^{\ast }$ is a point on a contracting $k$-periodic orbit
of the Poincare map $P$ and $p_{2}^{\ast }$ is a point on a hyperbolic $k$%
-periodic orbit of $P$. If there exists a curve $\gamma $, invariant under $%
P^{k}$ (i.e. $P^{k}\left( \gamma \right) =\gamma $), which joins $%
p_{1}^{\ast }$ with $p_{2}^{\ast }$, \correction{comment 1}{ such that $\bigcup_{i=1}^{k}  P^i(\gamma)$ is a closed curve, }then 
\begin{equation*}
\mathcal{T}:=\left\{ \Phi _{t}(v):v\in \gamma ,t\in \mathbb{R}\right\}
\end{equation*}%
is a two dimensional torus invariant under the flow of (\ref{eq:ode-3d}).
\end{lemma}
\begin{proof}
Continuing the curve $\bigcup_{i=1}^{k}  P^i(\gamma)$ along the flow gives a two dimensional torus, as required.
\end{proof}

%

\section{Applications} \label{sec:examples}
In this Section we apply our methods to two explicit examples. 
The first is the Van der Pol system
with periodic external forcing, where we prove the existence of smooth Lipschitz 
tori by means of the tools from Section \ref{sec:contractiing-maps}. The second example 
is an autonomous vector field introduced by Langford \cite{Langford} which exhibits a 
 \correction{comment 17}{Neimark-Sacker bifurcation}.
 For this system we establish the existence of $C^0$ tori by 
means of the tools from Section \ref{sec:homoclinic-maps}. An interesting aspect of the 
second example is that the tori are
neither differentiable nor Lipschitz, so that $C^0$ is in fact the most that can be established.
\correction{comment 37}{In all our computer assisted proofs we have used the CAPD\footnote{Computer Assisted Proofs in Dynamics: http://capd.ii.uj.edu.pl/} library.}

\subsection{Regular tori for the time dependent Van der Pol system}
In this Section we apply the methods from Sections \ref%
{sec:contractiing-maps}, \ref{sec:tori-3d} to establish the existence of $C^k$ and
Lipschitz tori in a periodically forced nonlinear oscillator.
For our example application we consider the Van der Pol \correction{comment 32}{equation}
with periodic forcing%
\begin{equation*}
x^{\prime \prime }-v(1-x^{2})x^{\prime }+x-\varepsilon \cos \left( t\right)
=0.
\end{equation*}
The system is a canonical example in dynamical systems theory going back to 
its introduction by Balthasar van der Pol in 1920 as a mathematical 
model for an electrical circuit containing a vacuum tube \cite{vanDerPol_1}.
For almost a century the system has been studied as a simple and 
physically relevant example of a differential equation 
exhibiting spontaneous nonlinear oscillations.
Later van der Pol himself considered the circuit when driven by a periodic
external forcing \cite{vanDerPolNature_forced}, and saw
what would today be called an attracting invariant torus.  
For a much more complete theoretical discussion of the 
dynamics of the forced Van der Pol system, 
as well as a thorough review of the literature,
we refer to the classic study of  \cite{gukenheimerVanDerPol}.
The interested reader is referred also to the works of    
\cite{deClassified,vanDerPolPlasma} for interesting applications of 
the forced system.

%
%
%
%

We prove the existence of a smooth and attracting invariant torus 
for the following pairs of parameters%
\begin{equation}
\left( v,\varepsilon \right) \in \left\{ \left( 0.1,0.002\right) ,\left(
0.2,0.005\right) ,\left( 0.3,0.01\right) ,\left( 0.4,0.015\right) ,\left(
0.5,0.05\right) ,\left( 1,0.1\right) \right\} .  \label{eq:VanDerPol-params}
\end{equation}%
To do so we consider the system in the extended
phase space%
\begin{eqnarray}
x^{\prime } &=&y,  \notag \\
y^{\prime } &=&v(1-x^{2})y-x+\varepsilon \cos \left( t\right) ,
\label{eq:Van-der-Pol-extended} \\
t^{\prime } &=&1,  \notag
\end{eqnarray}%
and take the time section $\Sigma =\left\{ t=0\right\} $. We consider
the time shift map $f^{v,\varepsilon }:\Sigma \rightarrow \Sigma $ defined
as $f^{v,\varepsilon }\left( x,y\right) =\Phi _{2\pi }^{v,\varepsilon
}\left( x,y,0\right) $, where $\Phi _{s}^{v,\varepsilon }$ is the flow
induced by (\ref{eq:Van-der-Pol-extended}).

For each pair $\left( v,\varepsilon \right) $ of parameters we take a
sequence of points $\left\{ p_{i}^{v,\varepsilon }\right\}
_{i=1}^{N^{v,\varepsilon }}\subset \Sigma $, which lie approximately on the
intersection of the torus with $\Sigma $. We draw these points in Figure \ref%
{fig:VenDerPol}. (These points are computed numerically. We choose them so
that they are roughly uniformly spread along the curves.)

We then choose local coordinates $\gamma _{i}^{v,\varepsilon }:\left[
-R_{i}^{v,\varepsilon },R_{i}^{v,\varepsilon }\right] \times \left[
-r^{v,\varepsilon },r^{v,\varepsilon }\right] \rightarrow \mathbb{R}^{2}$ as 
\begin{equation*}
\gamma _{i}^{v,\varepsilon }\left( x,y\right) =q_{i}^{v,\varepsilon
}+A_{i}^{v,\varepsilon }\left( 
\begin{array}{c}
x \\ 
y%
\end{array}%
\right) \qquad \text{for }i=1,\ldots ,N^{v,\varepsilon },
\end{equation*}%
with%
\begin{eqnarray*}
q_{1}^{v,\varepsilon } &=&\frac{1}{2}\left( p_{2}^{v,\varepsilon
}+p_{N}^{v,\varepsilon }\right) , \\
q_{N}^{v,\varepsilon } &=&\frac{1}{2}\left( p_{1}^{v,\varepsilon
}+p_{N-1}^{v,\varepsilon }\right) , \\
q_{i}^{v,\varepsilon } &=&\frac{1}{2}\left( p_{i+1}^{v,\varepsilon
}+p_{i-1}^{v,\varepsilon }\right) ,\qquad \text{for }i=2,\ldots
,N^{v,\varepsilon }-1,
\end{eqnarray*}\correction{comment 33}{
\begin{eqnarray*}
A_{1}^{v,\varepsilon } &=&\frac{1}{R_1^{v,\varepsilon }}\left( 
\begin{array}{ll}
\pi _{x}\frac{1}{2}\left( p_{2}^{v,\varepsilon }-p_{N}^{v,\varepsilon
}\right) & -\pi _{y}\frac{1}{2}\left( p_{2}^{v,\varepsilon
}-p_{N}^{v,\varepsilon }\right) \\ 
\pi _{y}\frac{1}{2}\left( p_{2}^{v,\varepsilon }-p_{N}^{v,\varepsilon
}\right) & \pi _{x}\frac{1}{2}\left( p_{2}^{v,\varepsilon
}-p_{N}^{v,\varepsilon }\right)%
\end{array}%
\right) , \\
A_{N}^{v,\varepsilon } &=&\frac{1}{R_{N}^{v,\varepsilon }}\left( 
\begin{array}{ll}
\pi _{x}\frac{1}{2}\left( p_{1}^{v,\varepsilon }-p_{N-1}^{v,\varepsilon
}\right) & -\pi _{y}\frac{1}{2}\left( p_{1}^{v,\varepsilon
}-p_{N-1}^{v,\varepsilon }\right) \\ 
\pi _{y}\frac{1}{2}\left( p_{1}^{v,\varepsilon }-p_{N-1}^{v,\varepsilon
}\right) & \pi _{x}\frac{1}{2}\left( p_{1}^{v,\varepsilon
}-p_{N-1}^{v,\varepsilon }\right)%
\end{array}%
\right) , \\
A_{i}^{v,\varepsilon } &=&\frac{1}{R_{i}^{v,\varepsilon }}\left( 
\begin{array}{ll}
\pi _{x}\frac{1}{2}\left( p_{i+1}^{v,\varepsilon }-p_{i-1}^{v,\varepsilon
}\right) & - \pi _{y}\frac{1}{2}\left( p_{i+1}^{v,\varepsilon
}-p_{i-1}^{v,\varepsilon }\right) \\ 
\pi _{y}\frac{1}{2}\left( p_{i+1}^{v,\varepsilon }-p_{i-1}^{v,\varepsilon
}\right) & \pi _{x}\frac{1}{2}\left( p_{i+1}^{v,\varepsilon
}-p_{i-1}^{v,\varepsilon }\right)%
\end{array}%
\right) ,
\end{eqnarray*}}
for $ i=2,\ldots ,N^{v,\varepsilon }-1,$ and%
\begin{eqnarray*}
R_{1}^{v,\varepsilon } &=&\left\Vert \frac{1}{2}\left( p_{2}^{v,\varepsilon
}-p_{N}^{v,\varepsilon }\right) \right\Vert , \\
R_{N}^{v,\varepsilon } &=&\left\Vert \frac{1}{2}\left( p_{1}^{v,\varepsilon
}-p_{N-1}^{v,\varepsilon }\right) \right\Vert , \\
R_{i}^{v,\varepsilon } &=&\left\Vert \frac{1}{2}\left(
p_{i+1}^{v,\varepsilon }-p_{i-1}^{v,\varepsilon }\right) \right\Vert ,\qquad 
\text{for }i=2,\ldots ,N^{v,\varepsilon }-1.
\end{eqnarray*}%
The choice is motivated by the fact that the set $\gamma _{i}\left( \left[
-R_{i}^{v,\varepsilon },R_{i}^{v,\varepsilon }\right] \times \left\{
0\right\} \right) $ is a line connecting the points $p_{i-1}^{v,\varepsilon
} $ and $p_{i+1}^{v,\varepsilon }$. As a result we obtain overlapping sets $%
U_{i}^{v,\varepsilon }:=\gamma _{i}\left( \left[ -R_{i}^{v,\varepsilon
},R_{i}^{v,\varepsilon }\right] \times \left[ -r^{v,\varepsilon
},r^{v,\varepsilon }\right] \right) $ which cover the true invariant curve
for $f^{v,\varepsilon }$. We establish the existence of the curve using the
method outlined in Section \ref{sec:Lip-curves-validation}. We outline some of the aspects of our computer assisted proof below.

For the first five parameters from (\ref{eq:VanDerPol-params}) it turned out
to be enough to consider $N^{v,\varepsilon }=1000$. For these five
parameters we have chosen $r^{v,\varepsilon }=5\cdot 10^{-4}$, and we have
chosen the slope of the cones (\ref{eq:cones-Lip-torus}) as $a_{i}=0.3$. Each set $U_i$ was additionally subdivided into
 $m=6$ parts for the validation of condition (\ref{eq:fragment-contraction}).
This condition was the most time consuming part of the proof. The computer
assisted proof for each of the four tori took under a minute and a half on a
single 3 GHz Intel i7 Core processor. \correction{comment 34}{(The parameters $r_i,a_i$ and $m$ were chosen by trial and error.)}

The validation of (\ref{eq:fragment-contraction}) was based on (\ref%
{eq:contraction-lip-curve-validation}) so as a 
\correction{comment 15}{byproduct}
 from (\ref%
{eq:local-map-derivative}) we obtained bounds on the derivatives of the
local maps, which allows us to compute the bounds $\mu ,\xi $ needed for
the rate conditions (see Definition \ref{def:rate-cond}). By using Theorem %
\ref{th:Ck} we validate the $C^{k}$ regularity for the torus at parameter
pairs $\left( v,\varepsilon \right) =\left( 0.1,0.002\right) ,\left(
0.2,0.005\right) ,\left( 0.3,0.01\right) $ as $k=9,5,2$, respectively. For
the remaining parameters from (\ref{eq:VanDerPol-params}) we only obtained
that the tori are Lipschitz. This is due to the fact that the higher the $v$
the less `uniform' the dynamics on the torus. What we mean by this is that
there are regions on the torus in which the dynamics restricted to the torus
is expanding or contracting (the torus does not behave uniformly as a
central coordinate). This affects the bounds on parameters $\mu ,\xi $ (the
second parameter in particular) which results in weaker regularity bounds
obtained from our method.

Non-uniformity of the dynamics on the torus for higher $v$
has also made the proof for the parameters $\left( v,\varepsilon \right)
=\left( 1,0.1\right) $ more computationally demanding. We  take $%
N^{v,\varepsilon }=5000$ and $m=20$, covering the curve with
a larger number of fragments. We also take $r^{v,\varepsilon }=2\cdot
10^{-5} $, and the slope of the cones (\ref{eq:cones-Lip-torus}) were taken
as $a_{i}=0.1$. With $a_i$ a larger number, but using smaller 
sets $U_{i}$ allows us to validate the needed conditions in this example.
The computer assisted proof
for this parameter pair took 31 minutes on a single 3GHz Intel i7 Core
processor.

This demonstrates the following weakness of our method. It performs well if
the dynamics on the torus is uniform. If it is not, then proofs require many
subdivisions. In the next Section we consider another example in which this
problem is even more visible.

\correction{comment 5}{
\begin{remark}
In the computer assisted proof we can use a small interval of parameters instead of a single parameter value. By invoking parallel computations on a cluster, one could use our approach to cover whole parameter ranges.
\end{remark}}

We finish with the comment that by Theorem \ref{th:Lip-torus} the invariant curves established for the map $%
f^{v,\varepsilon }$ lead to two dimensional invariant tori of (\ref%
{eq:Van-der-Pol-extended}).

\begin{figure}[tbp]
\begin{center}
\includegraphics[height=5.5cm]{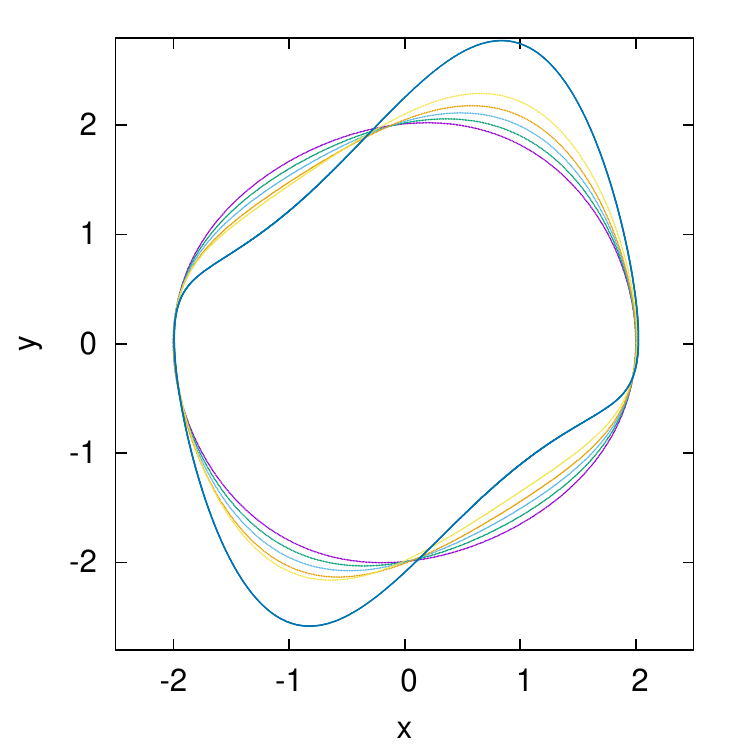}
\end{center}
\caption{{}Intersections of the invariant Lipschitz tori for the Van der Pol
system with the $t = 0$ section for each parameter from (\protect\ref{eq:VanDerPol-params}). 
The smaller the $\protect\mu $ the more circular/smooth the curve.}
\label{fig:VenDerPol}
\end{figure}

\subsection{Resonant tori in the Langford system}

Consider the autonomous vector field $F:\mathbb{R}%
^{3}\rightarrow \mathbb{R}^{3}$ given by the formula 
\begin{equation*}
F(x,y,z)=%
\begin{pmatrix}
(z-\beta )x-\delta y \\ 
\delta x+(z-\beta )y \\ 
\gamma +\alpha z-\frac{z^{3}}{3}-(x^{2}+y^{2})(1+\epsilon z)+\zeta zx^{3}%
\end{pmatrix}%
,
\end{equation*}%
where $\epsilon =0.25$, $\gamma =0.6$, $\delta =3.5$, $\beta =0.7$, $\zeta
=0.1$, and $\alpha =0.95$ are the `classical' parameter
values. This system is 
a toy model for dissipative vortex dynamics, or for a rotating viscus fluid,
and was first studied by Langford in \cite{Langford}. 
We define a section $\Sigma =\left\{ x=0\right\} $, and the first return
time section to section map $P:\Sigma \rightarrow \Sigma $. 


We treat the parameter $\alpha $ as our bifurcation parameter. For all 
$\alpha \in \lbrack 0,0.95]$ there exists a fixed point in $\Sigma $ of $%
P^{2}$ which corresponds to a periodic orbit $\tau $ of the ODE. The
periodic orbit has complex conjugate Floquet multipliers which are stable
for small $\alpha $ but which later cross the unit circle, loosing stability
in a \correction{comment 17}{Neimark-Sacker bifurcation} \cite{Aizawa}, which occurs at $%
\alpha \approx 0.69714$ and gives birth to a $C^{k}$ torus. We give a plot
of such torus for one of the parameters in Figure \ref{fig:Aizawa075}.

\begin{figure}[tbp]
\begin{center}
\includegraphics[width=1\textwidth,height=4cm]{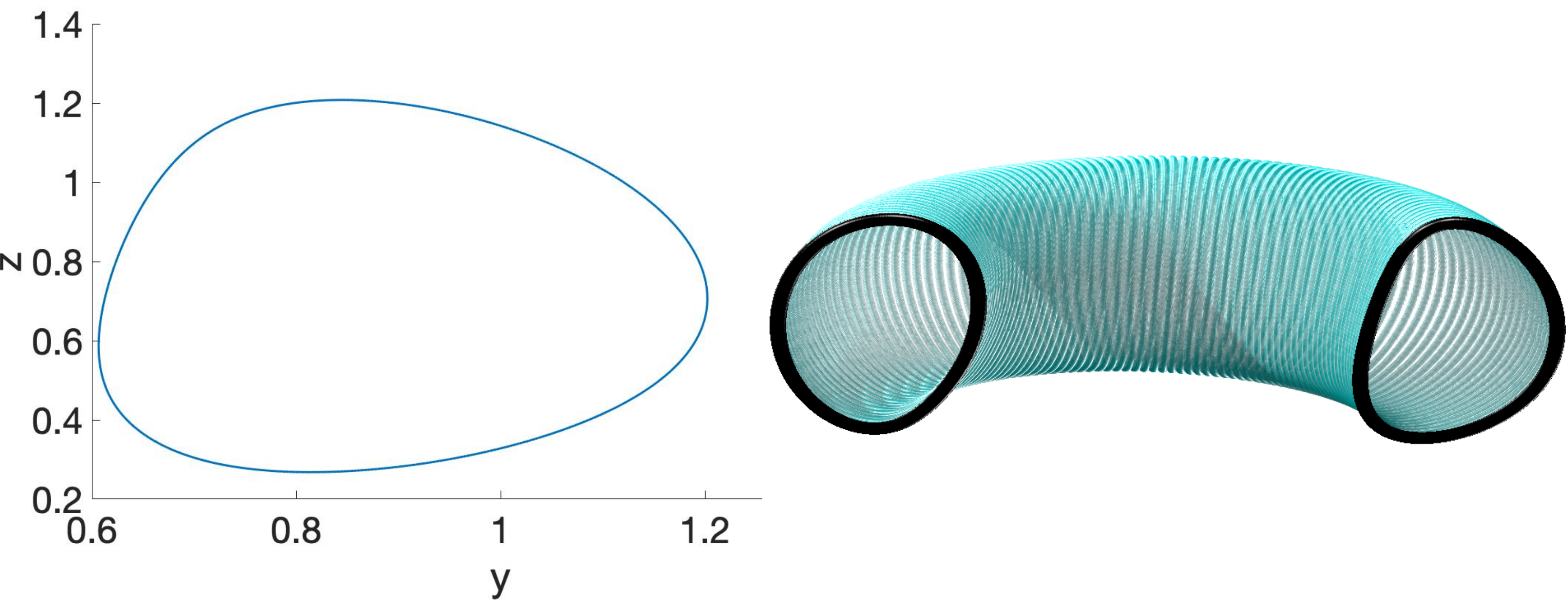}
\end{center}
\caption{At $\protect\alpha =0.75$ we have an attracting limit cycle of $P^2$ on $%
\Sigma $ (figure on the left) which is the intersection of the two
dimensional $C^k$ torus of the ODE with $\Sigma \cap \{y>0\}$. On the right
we plot half of the torus. In black we have both components of the torus
intersection with $\Sigma$; one for $y<0$ and the other for $y>0$.}
\label{fig:Aizawa075}
\end{figure}


Additionally there exist two period six orbits of $P$ in $\Sigma .$ One is a
saddle periodic orbit which we denote as $c_{h}$, and the other is a stable
focus periodic orbit, which we denote as $c_{s}$. (We use the subscript $h$ to
stand for `hyperbolic' and $s$ to stand for `stable'.) We found that one
branch of $W^{u}(c_{h})$ wraps around the torus, while the other reaches $%
c_{s}$; see left plot in Figure \ref{fig:Aizawa4}. This happens right until $%
\alpha \approx 0.822$.

As we increase further our bifurcation parameter $\alpha $, our invariant
two-dimensional $C^{k}$ torus bifurcates to a 
$C^{0}$ torus. This bifurcation happens by $c_{h},c_{s}$ colliding with the
torus. Another way of interpreting this is by looking at what happens with $%
W^{u}(c_{h})$ and $W^{s}(c_{h})$. Before the bifurcation one branch of $%
W^{u}(c_{h})$ goes inside, wrapping around the torus; see left plot from
Figure \ref{fig:Aizawa3}. After the bifurcation both branches of $%
W^{u}(c_{h})$ lead to $c_{s}$; see right plot from Figure \ref{fig:Aizawa3}
and also Figure \ref{fig:Aizawa5}. Since after the bifurcation the tori
include the periodic orbits, we refer to them as \emph{resonant tori}.

\begin{figure}[tbp]
\begin{center}
\includegraphics[width=0.485\textwidth,height
=0.42\textwidth]{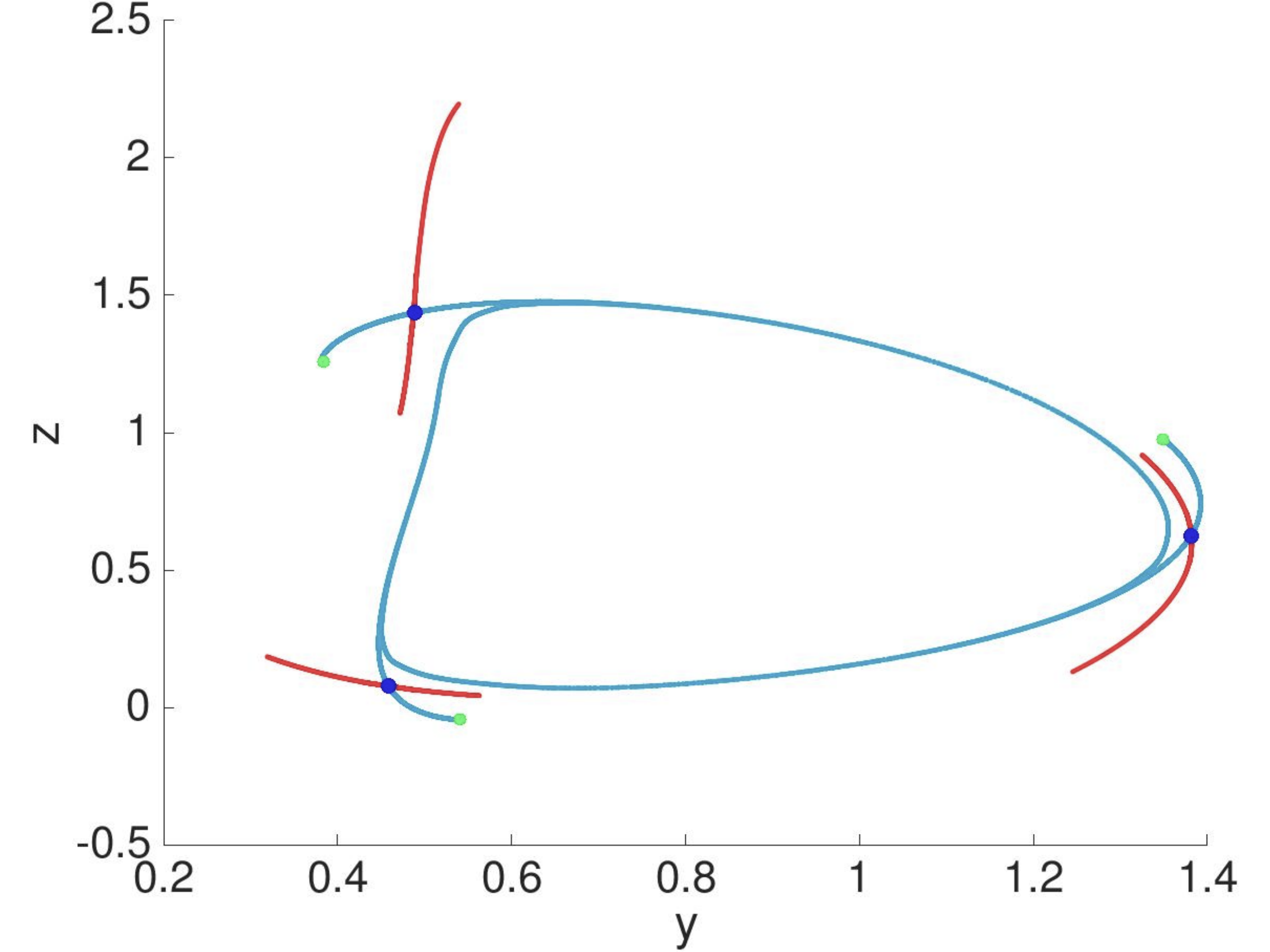}\quad \includegraphics[width=0.485%
\textwidth,height =0.42\textwidth]{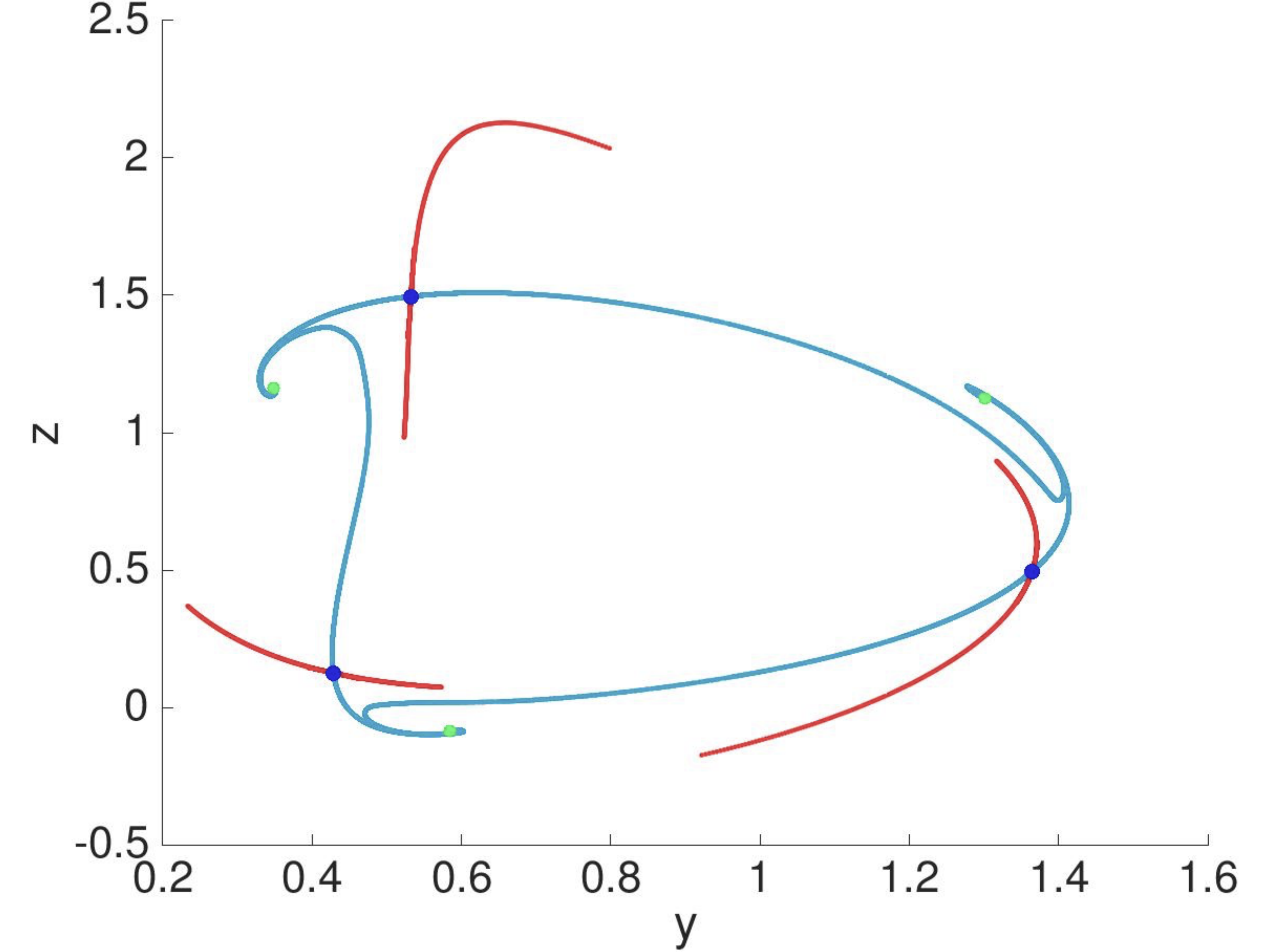}
\end{center}
\caption{The plot of the periodic orbits $c_{h}$ and $c_{s}$ on $\Sigma \cap
\{y>0\}$. (The orbits are of period $6$, but we plot only half of the points
with $y>0$.) The hyperbolic orbit $c_{u}$ is in blue, and the attracting
orbit $c_{s}$ is in green. The manifold $W^{s}(c_{h})$ is in red and $%
W^{u}(c_{h})$ is blue. (Left) at $\protect\alpha =0.815$, we see that a
branch of $W^{u}(c_{h})$ goes inside and wraps around the attracting
invariant circle. (Right) at $\protect\alpha =0.835$, we observe that a
branch of $W^{u}(c_{h})$ goes to the other side of $W^{s}(c_{h})$ and gets
caught in the basin of attraction of $c_{s}$.}
\label{fig:Aizawa3}
\end{figure}


\begin{figure}[tbp]
\begin{center}
\includegraphics[width=1\textwidth,height
=0.52\textwidth]{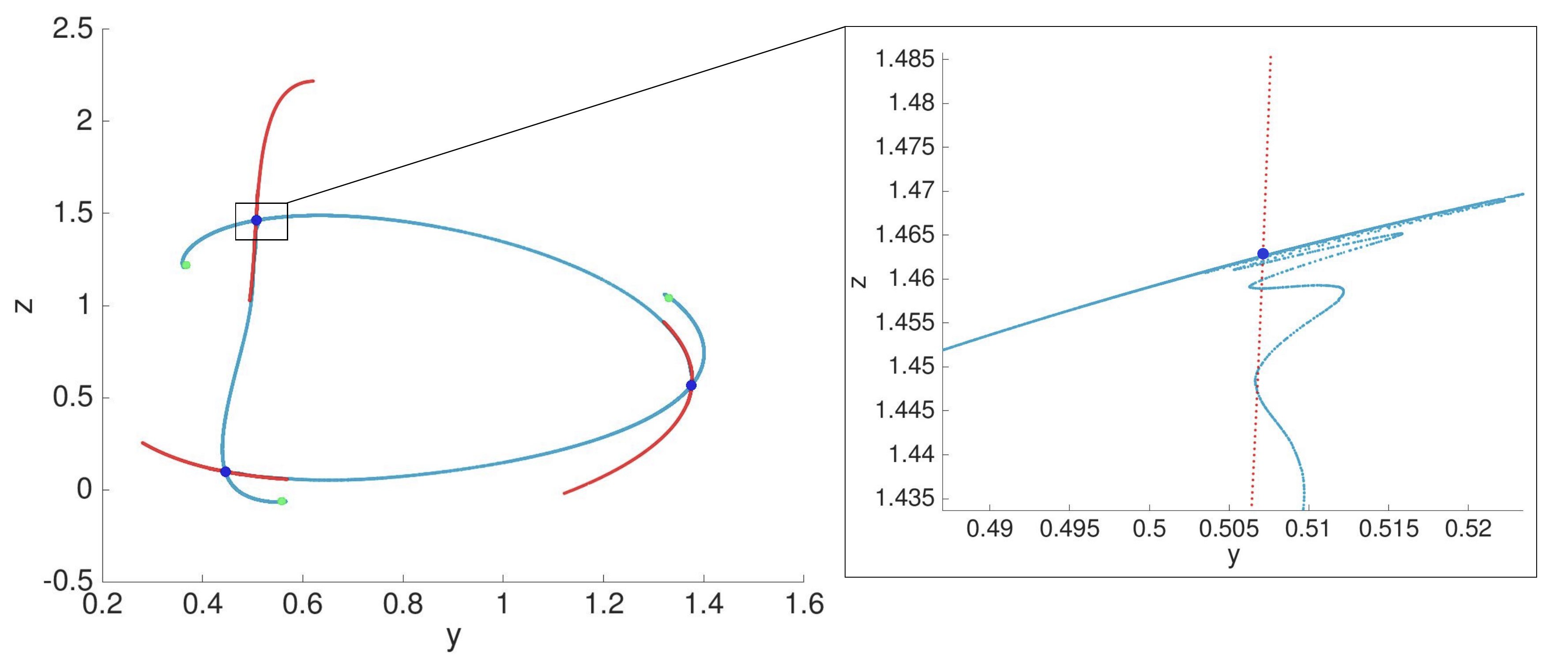}
\end{center}
\caption{Colors have the same meaning as in Figure \protect\ref{fig:Aizawa3}%
. At $\protect\alpha =0.8225$ we have transverse intersections of $%
W^{u}(c_{h})$ and $W^{s}(c_{h})$ which leads to chaotic dynamics.}
\label{fig:Aizawa4}
\end{figure}

For parameters between the case of $C^{k}$ tori and the case of resonant
tori we have transverse intersections of $W^{u}(c_{h})$ and $W^{s}(c_{h})$
as seen in Figure \ref{fig:Aizawa4}. This transverse intersection leads to
the presence of Smale horseshoes and thus chaotic dynamics. These transverse
intersections are born and terminated at parameters for which we have
tangential intersections of $W^{s}(c_{h})$ with $W^{u}(c_{h})$.



\begin{figure}[tbp]
\begin{center}
\includegraphics[width=1\textwidth,height =0.4\textwidth]{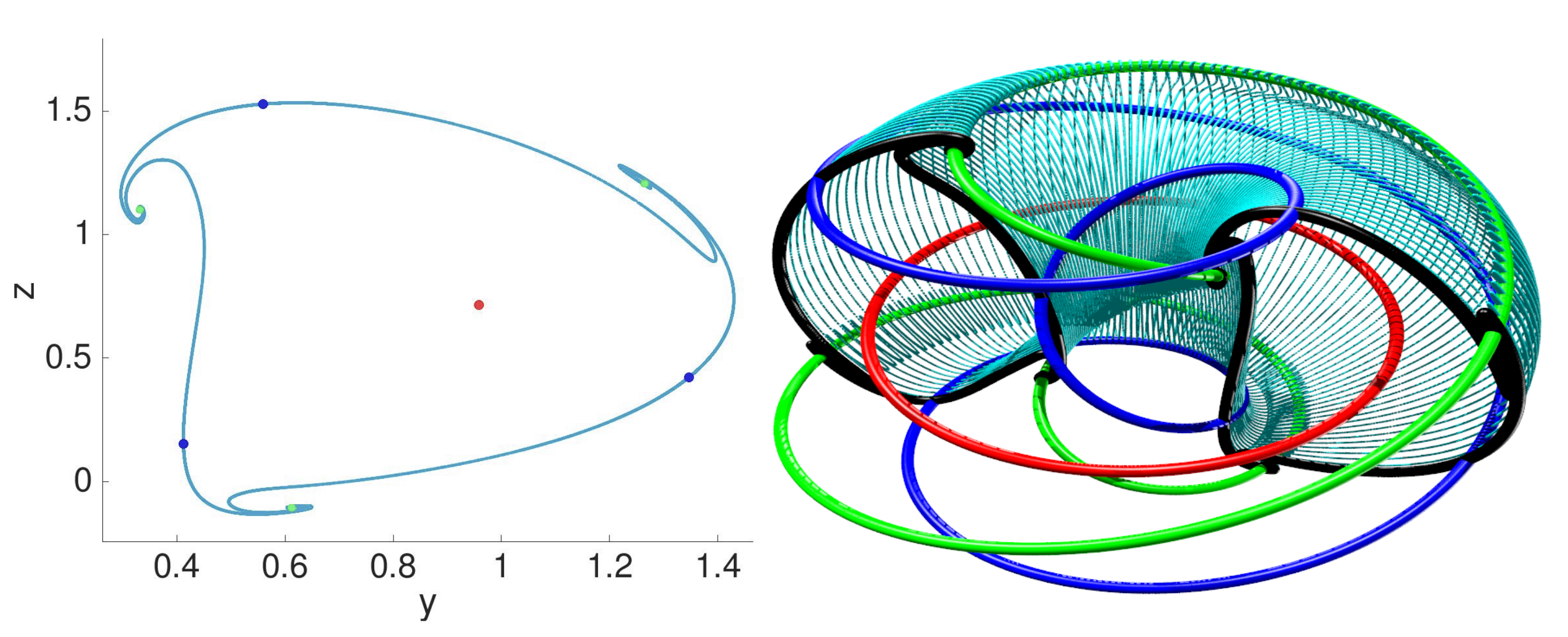}
\end{center}
\caption{A resonance torus for $\protect\alpha =0.85$. In red we plot the
periodic orbit from which the tori have initially originated through the
Hopf type bifurcation. }
\label{fig:Aizawa5}
\end{figure}



Our objective is to apply the tools from Section \ref%
{sec:heteroclinic-subsection} and prove the existence of a resonant tori. Below we
describe the proof when 
$\alpha =0.85$.  The resonant torus in our proof 
is depicted in Figure \ref{fig:Aizawa5}. 

\begin{remark}
We emphasize that the resonant torus from Figure \ref{fig:Aizawa5} is
continuous, but \emph{not} $C^{k}$ for $k>0$. In fact, it is not even
Lipschitz due to the rotation around the attractinmg periodic orbit.
\end{remark}

We now outline the details of the proof. 
Consider the map $f:\mathbb{R}^{2}\rightarrow \mathbb{R}^{2}$ 
\begin{equation*}
f:=P^{6}.
\end{equation*}%
First we find bounds on two fixed points $%
p_{1}^{\ast },p_{2}^{\ast }\in \Sigma $ of $f$, for which $p_{1}^{\ast }\in
c_{s}$ and $p_{2}^{\ast }\in c_{h}$. We do this using the following shooting
method. Consider%
\begin{equation*}
G:\underset{6}{\underbrace{\mathbb{R}^{2}\times \ldots \times \mathbb{R}^{2}%
}}\rightarrow \mathbb{R}^{12},
\end{equation*}%
defined as%
\begin{equation*}
G\left( q_{1},\ldots ,q_{6}\right) :=\left( P\left( q_{6}\right)
-q_{1},P\left( q_{1}\right) -q_{2},P\left( q_{2}\right) -q_{3},\ldots
,P\left( q_{5}\right) -q_{6}\right) .
\end{equation*}%
(Note that $q_{i}=\left( y_{i},z_{i}\right) \in \Sigma $, for $i=1,\ldots 6$%
.) Establishing that%
\begin{equation*}
G\left( q_{1},\ldots ,q_{6}\right) =0,
\end{equation*}%
gives points $q_{1},\ldots ,q_{6}$ on a period $6$ orbit of $P$.

We using the interval Newton method (Theorem \ref{th:interval-Newton}) to
validate that the point $\left( q_{1},\ldots ,q_{6}\right) $ is in $%
\prod_{i=1}^{12}I_{i}$, for some closed intervals $I_{i}$, for $i=1,\ldots
,12$. Once this is done, the two dimensional box $I_{1}\times I_{2}$ is
an enclosure of $q_{1}$ which is a fixed point of $P^{6}.$ We find that
\begin{eqnarray*}
p_{1}^{\ast } &\in &\left( 
\begin{array}{r}
0.611160359286522+4.6\cdot 10^{-13}\cdot \left[ -1,1\right] \\ 
-0.104496536895459+8.8\cdot 10^{-13}\cdot \left[ -1,1\right]%
\end{array}%
\right) , \\
p_{2}^{\ast } &\in &\left( 
\begin{array}{c}
0.413216691560642+2.5\cdot 10^{-13}\cdot \left[ -1,1\right] \\ 
0.150271844775546+9.4\cdot 10^{-13}\cdot \left[ -1,1\right]%
\end{array}%
\right) .
\end{eqnarray*}

Now take the two $2\times 2$ matrices 
\begin{eqnarray*}
A_{1} &:&=\left( 
\begin{array}{ll}
0.953174 & -0.0468255 \\ 
0.169128 & 0.2639%
\end{array}%
\right) , \\
A_{2} &:&=\left( 
\begin{array}{ll}
0.0138304 & -1 \\ 
-1 & 0.674926%
\end{array}%
\right),
\end{eqnarray*}%
and define the local maps $f_{1}$ and $f_{2}$ around $p_{1}^{\ast }$
and $p_{2}^{\ast }$, respectively, as%
\begin{equation*}
f_{i}\left( q\right) :=A_{i}^{-1}\left( f\left( A_{i}q+p_{i}^{\ast }\right)
-p_{i}^{\ast }\right) \qquad \text{for }i=1,2.
\end{equation*}%
Note that $p_{1}^{\ast }$ and $p_{2}^{\ast }$ are shifted to zero in their
respective local coordinates. The choices of $A_{1}$ and $A_{2}$ are such
that they put the derivatives of $f_{1}$ and $f_{2},$ respectively, at zero
approximately in Jordan form. Such $A_{1}$ and $A_{2}$ are computed using standard
numerics (we do not need interval arithmetic validation at this stage). 

We consider the two cubes $B^{1}$ and $B^{2}$ in $\mathbb{R}^{2}$ defined by 
\begin{eqnarray*}
B^{1} &:=&[-0.0005,0.0005]\times \lbrack -0.0005,0.0005], \\
B^{2} &:=&[-0.0001,0.0001]\times \lbrack -2\cdot 10^{-8},2\cdot 10^{-8}].
\end{eqnarray*}%
With computer assistance we established that zero is an attracting fixed
point of $f_{1}$ in $B^{1}$. This was done in interval arithmetic by using
Lemma \ref{lem:contraction-main}. We also established that the unstable
manifold of zero for the map $f_{2}$ is contained in the cone $Q(0)$ of the
form (\ref{eq:cone-def}) with $L=2\cdot 10^{-4}$. We did this by using Lemma %
\ref{lem:wu}. The validation of the assumptions of Lemmas \ref%
{lem:contraction-main} and \ref{lem:wu} was based on interval arithmetic
bounds on the derivative of the map. Here we write out the bounds we have obtained: 
\begin{eqnarray*}
\left[ Df_{1}\left( B^{1}\right) \right] &=&\left( 
\begin{array}{ll}
\lbrack 0.150243,0.220614] & [-0.561824,-0.521109] \\ 
\lbrack 0.41934,0.663593] & [0.10723,0.263629]%
\end{array}%
\right) , \\
\left[ Df_{2}\left( B^{2}\right) \right] &=&\left( 
\begin{array}{ll}
\lbrack 2.16813,2.16975] & [-0.000485,0.000485] \\ 
\lbrack -0.000352,0.000351] & [0.195584,0.195806]%
\end{array}%
\right) .
\end{eqnarray*}

We now consider%
\begin{equation*}
\bar{x}=4.5\cdot 10^{-5}.
\end{equation*}%
With computer assistance we have validated that $\pi _{u}f_{2}\left(
Q(0) \cap \{p:\pi_x p=\bar{x}\}\right) \subset I:=[\bar{x},0.0001]$ and that for $%
n=25$ 
\begin{equation}
A_{1}^{-1}\left( f^{n}\left( A_{2} \left(
Q(0) \cap \{p:\pi_x p\in I\}\right)
+p_{2}^{\ast }\right) -p_{1}^{\ast }\right) \subset B^{1}.
\label{eq:fundamental-domain-condition}
\end{equation}%
This by Theorem \ref{th:connecting-curve} establishes the existence of an
invariant curve for $f$, which joins $p_{1}^{\ast }$ and $p_{2}^{\ast }$.
The resonant torus from Figure \ref{fig:Aizawa5} follows from Lemma
 \ref{lem:resonant-torus}.

\begin{remark}
The condition (\ref{eq:fundamental-domain-condition}) required $n=25$
iterates of the map $f$, which is $6n=150$ iterates of the map $P$; this
requires a long integration time of the ODE. This was the most time
consuming part of the computer assisted proof, since it required a
subdivision of $Q(0) \cap \{p:\pi_x p\in I\}$ into $200$ fragments and checking (%
\ref{eq:fundamental-domain-condition}) for each of them separately.
\end{remark}

The computer assisted proof of the resonant torus for $\alpha =0.85$ took
under $6$ minutes on a single 3GHz Intel i7 Core processor.

There is nothing particularly special about the parameter $\alpha =0.85$.
Using the same techniques, we have obtained proofs of resonant tori for other parameters,
including $\alpha =0.835$ for which we have the plot of the torus in the
right hand plot from Figure \ref{fig:Aizawa3}.

We finish this Section by commenting on difficulties we have encountered
when trying to validate the $C^{k}$ tori for smaller parameters $\alpha $.
We ran into these when considering for instance $\alpha =0.75$ for which the
torus is plotted in Figure \ref{fig:Aizawa075}. Judging by the shape of the
torus it would seem to be well suited for the validation
methods of Section \ref{sec:contractiing-maps}. Our problem in this
particular example is that the dynamics near the torus are not uniformly 
contracting. There
are some regions of expansion, and other regions of strong
contraction. In total the torus is an attractor, but it is not a uniform one
and the methods of Section \ref{sec:contractiing-maps} do not
apply. When $\alpha =0.75$ such uniform contraction is
achieved for $f=P^{16}$. We have been able to enclose the curve in a set $U$
which consists of $10000$ cubes and validate that $f$ is contracting in $U$
and that $f(U)\subset U$. (Such validation has been very time consuming and
took 5 hours and 27 minutes on a single 3GHz Intel i7 Core
processor.) This establishes the existence of an invariant set in $U$, but
does not prove that this set is a torus. Using the results from \cite{CK}
one obtains that this invariant set projects surjectively onto a torus, but
other than this we do not get any information about its topology.

To prove that the invariant set is a torus we would need to also validate
cone conditions. The fact that $f$ consists of $16$ iterates of $P$ leads to
long integration of the ODE. This resulted in insufficiently sharp estimates
on the derivative of $f$ and we were unable to validate cone conditions. An
additional difficulty we have encountered is that in the neighborhood of
the invariant curve of $f$ we do not have `vertical' contraction towards the
curve, but also strong twist dynamics. This makes the validation of cone
conditions even harder, since the angle between the center and the stable
bundles becomes very small. 

\correction{comment 36}{
To be more precise, if on the section $\Sigma$ we choose the tangent vector and 
the normal vector to the invariant circle as the basis for our coordinates, then for 
a point $p$ from the invariant circle the derivative of $f=P^{16}$ is of the form
\[
Df(p)=\left(
\begin{array}
[c]{cc}%
1 & \delta\\
0 & \lambda
\end{array}
\right),
\]
with $\left\vert \lambda\right\vert <1$, but $\left\vert \lambda\right\vert
\approx1$, and $\left\vert \delta\right\vert \gg1.$ This means that in order
to validate cone conditions, the fact that $(x,y)\in Q(0)$, i.e. $\left\vert
y\right\vert <a\left\vert x\right\vert $, should imply that $Df(p)(x,y)\in
Q(0)$, i.e. $\left\vert \lambda y\right\vert < a \left\vert x+y\delta\right\vert
$. The choice of $y=-\delta^{-1}x$ will result in zero on the right hand side,
which means that a necessary condition is to have $a\leq\left\vert
\delta\right\vert ^{-1}$. This means that in the case when $\delta$ is a large
number (we have a strong twist), we have to choose small $a$, which means that we need to use sharp cones. The smaller the $a$ the more difficult is the validation
of cone conditions. On top of that, we need a large number of iterates of $P$ to compute $f$, which leads to long integration times, resulting
in insufficiently accurate bounds on the derivatives of $f$ in order to
validate the cone conditions.
}

\correction{}{
We encounter exactly the same problem when the parameters of the
system are close to the Neimark-Sacker bifurcation. In such setting the torus is not strongly attracting, meaning that a large number of iterates of $P$ is needed for the contraction to be strong enough so that we can validate $f(U)\subset U$. This results in the appearance of a large twist parameter $\delta$ in derivatives of $f$, and we run into identical problems as those described above. }

This demonstrates that our method has limitations
in the presence of twist and nonuniform contraction of the invariant tori. Developing 
a computer assisted proof strategy which overcomes these difficulties would be an 
interesting future project.  Another interesting project would be to formulate functional 
analytic methods for studying rotational invariant tori which could possibly lead to
sharp or sharper regularity bounds.  


\section{Acknowledgements}
We would like to thank the anonymous Reviewers for their comments, suggestions and corrections, which helped us improve our paper.

\bibliographystyle{unsrt}
\bibliography{refs}

\medskip
Received xxxx 20xx; revised xxxx 20xx.
\medskip

\end{document}